\documentclass[11pt]{article}
\usepackage{bbm}
\usepackage{geometry}
\usepackage{amsfonts}
\usepackage{latexsym, amssymb, amsmath, amscd, amsthm, amsxtra}
\usepackage{mathtools}
\usepackage{enumerate}
\usepackage[all]{xy}
\usepackage{mathrsfs}
\usepackage{fancyhdr}
\usepackage{listings}
\usepackage[hidelinks=true]{hyperref}
\usepackage{graphicx}
\usepackage{tikz}
\def\P{\mathbb{P}}
\def\Pr{\mathrm{\mathbf{P}}}
\def\Ex{\mathrm{\mathbf{E}}}
\def\E{\mathbb{E}}

\def\Z{\mathbb{Z}}
\def\R{\mathbb{R}}
\def\N{\mathbb{N}}

\def\rad{\varrho_n}
\def\epn{{\varepsilon_n}}
\def\btn{\lambda_*}

\def\hbn{{\widehat{\mathcal B}_n}}
\def\qbn{{\hbn}}
\def\cec{\epsilon}

\def\mb{\mu_B}

\def \tp{\mathcal B}
\def \tcp{\mathcal B_\circ}
\def\11{\mathbbm{1}}

\def\tran{\mathsf{T}}

\def\lj{\mathcal J}

\def\lm{\ell}
\def\emm{\epsilon}

\def\ob{\mathcal{O}}

\def\plr{\mathcal V}%poly-log region
\def\tos{\mathcal T}

\def\epm{{\epsilon}}
\renewcommand{\epsilon}{\varepsilon}
\def\emp{\mathcal E}

\def\fregion{{\mathcal U}}
\def\lbs{\lambda_*}
\def \consa {b}

\newcommand{\eve}[2]{\Phi_{#1}(#2)}
\newcommand{\eva}[2]{\lambda_{#1}(#2)}
\DeclareFontFamily{U}{mathc}{}
\DeclareFontShape{U}{mathc}{m}{it}{<->s*[1.3] mathc10}{}
\DeclareMathAlphabet{\mymathcal}{U}{mathc}{m}{it}
\newcommand{\ct}[1]{\mymathcal{x}_{#1}}

\def\bo{\rm o}
\def\be{\rm e}

\newcommand\dif{\mathop{}\!\mathrm{d}}
\newcommand\capa{\mathrm{cap}}

\newcommand{\trueop}{$\ell$-``truly''-open}

\newtheorem{thm}{Theorem}[section]

\newtheorem{proposition}[thm]{Proposition}

\newtheorem{lemma}[thm]{Lemma}
\newtheorem{cor}[thm]{Corollary}

\newtheorem{claim}[thm]{Claim}

\newtheorem{thmx}{Theorem}

\theoremstyle{definition}
\newtheorem{defn}[thm]{Definition}
\newtheorem{remark}[thm]{Remark}

\numberwithin{equation}{section}
\begin{document}
\definecolor{qing}{RGB}{0, 153, 153}
\newcommand{\qing}[1]{\textcolor{qing}{#1}}
\newcommand{\rf}{}
\newcommand{\sun}[1]{\textcolor{cyan}{(Sun: #1)}}

\title{Distribution of the random walk conditioned on survival among quenched Bernoulli obstacles}
\author{Jian Ding$^{\,1}$ \and Ryoki Fukushima$^{\,2}$ \and Rongfeng Sun$^{\,3}$ \and Changji Xu$^{\,4}$}

\date{\today}

\maketitle

\footnotetext[1]{Statistics Department, University of Pennsylvania. Email: dingjian@wharton.upenn.edu}

\footnotetext[2]{Research Institute for Mathematical Sciences, Kyoto University. Email: ryoki@kurims.kyoto-u.ac.jp}

\footnotetext[3]{Department of Mathematics, National University of Singapore. Email: matsr@nus.edu.sg}

\footnotetext[4]{Department of Statistics, University of Chicago. Email: changjixu@galton.uchicago.edu}

\begin{abstract}
Place an obstacle with probability $1-p$ independently at each vertex of $\Z^d$ and consider a simple symmetric random walk that is killed upon hitting one of the obstacles. For $d \geq 2$ and $p$ strictly above the critical threshold for site percolation, we condition on the environment such that the origin is contained in an infinite connected component free of obstacles. It has previously been shown that with high probability, the random walk conditioned on survival up to time $n$ will be localized in a ball of volume asymptotically $d\log_{1/p}n$. In this work, we prove that this ball is free of obstacles, and we derive the limiting one-time distributions of the random walk conditioned on survival. Our proof is based on obstacle modifications and estimates on how such modifications affect the probability of the obstacle configurations as well as their associated Dirichlet eigenvalues, which is of independent interest.
\end{abstract}

\vspace{.3cm}


\vspace{12pt}

\tableofcontents

\medskip

\noindent

\section{Introduction}
\subsection{Model and main results}
For $d \geq 2$, let $(S_t)_{t \geq 0}$ be a discrete time simple symmetric random walk on $\Z^d$ , with $\Pr^x$ and $\Ex^x$ denoting probability and expectation for the random walk with $S_0 = x \in \Z^d$ and the superscript omitted when $x$ is the origin. We place the random walk in a random environment where an obstacle is placed independently at each point $x \in \Z^d$ with probability $1 - p \in (0,1)$, with $\P$ and $\E$ denoting probability and expectation for the random environment. We will say $x$ is {\em closed}  if $x$ is occupied by an obstacle, and $x$ is {\em open}  otherwise. Denote by $\ob$ the set of sites occupied by the  obstacles. The random walk is killed at the moment it hits an obstacle, namely at the stopping time
\begin{equation}
	\label{eq:def-tau}
	\tau := \tau_\ob= \inf\{ t \geq 0: S_t \in \ob \}\,.
\end{equation}
More generally, we denote by $\tau_A$ the first hitting time of a set $A\subset\Z^d$.
We will assume $p > p_c(\Z^d)$, the critical threshold for site percolation, and let $\widehat \P$ be the conditional probability measure for $\ob$ given that the origin is in the infinite open cluster. Given an environment under $\widehat \P$, we are interested in the behavior of the random walk given that it survives for a long time.

Recently, it has been shown in \cite{DX17,DX18} that conditioned on survival up to time $n$, the random walk stays in an \emph{island} (determined by the environment)  of diameter at most poly-logarithmic in $n$ during time $[o(n),n]$. Furthermore, at any deterministic time $t \in [o(n),n]$, the random walk stays with high probability in a ball of radius asymptotically
\begin{equation}
\label{eq:def-r}
	\rad = \lfloor (\omega_d^{-1}d\log_{1/p} n)^{1/d} \rfloor\,,
\end{equation}
where $\omega_d$ is the volume of the unit ball in $\R^d$. Namely, the following was shown in \cite{DX17,DX18}.
\begin{thmx}
\label{ballthm}
There exist a constant $C = C(d,p)$, %and
$\mymathcal {x}_n=\mymathcal {x}_n(\ob)\in\Z^d$ within distance $C(\log n)^{-2/d}n$ from the origin, and $\epsilon_n>0$ tends to 0 as $n\to \infty$ such that %if $B_n$ is the discrete ball of radius $(1+\epsilon_n)\rad$ (where ) centered at $\mymathcal {x}_n $, then
\begin{equation}
\label{eq:ball - localization}
    \min_{C |\mymathcal {x}_n| \leq t \leq n}\Pr(S_t  \in B(\mymathcal {x}_n,(1+\epsilon_n)\rad) \mid \tau>n) \to 1 \mbox{ in } \widehat\P\mbox{-probability} \,,
\end{equation}
where $B(x,r)$ denotes the Euclidean ball with center $x$ and radius $r$.
\end{thmx}
\noindent This improves earlier results which we will briefly review in Section~\ref{sec:literature}.

The study of the random walk killed by random obstacles is partially motivated by its relation to the so-called \emph{Anderson localization}. Let us briefly explain this relation. The generator of the killed random walk can be formally written (modulo a minus sign) as the random Schr\"{o}dinger operator $-\frac{1}{2d}\Delta+\infty\cdot \11_{\ob}$, where $\Delta$ is the discrete Laplacian. For this type of operators, various localization phenomena have been predicted and some of them have been rigorously proved; see e.g., \cite{AW15,Wolfgang16}. One formulation is that the eigenfunctions corresponding to small eigenvalues are supported on well-separated small regions where the potential takes atypically small values. Our problem has closer connection to the parabolic initial-boundary value problem
\begin{equation}
\label{eq:PAM-killing}
\begin{cases}
u(n+1,x)-u(n,x)=\frac{1}{2d} \Delta u(n,x), &(n,x)\in \Z_+\times (\Z^d\setminus \ob), \\
u(n,x)=0, &(n,x)\in \Z_+\times \ob,\\
u(0,x)=\11_{\{0\}}(x),
\end{cases}
\end{equation}
as the probability $\Pr(S_n=x, \tau>n)$ represents its unique bounded solution. Since $\Pr(\tau>n)$ is the total mass of the solution, the conditional probability $\Pr(S_n=x \mid \tau>n)$ is the normalized mass distribution. Now by the eigenfunction expansion, the long time asymptotics is determined by the small eigenvalues of $-\frac{1}{2d}\Delta+\infty\cdot \11_{\ob}$. Further taking into account the above eigenfunction formulation of the localization, one expects that the profile consists of well-separated peaks and each of them looks like an eigenfunction. While there are non-rigorous steps in this argument, Theorem~\ref{ballthm} proves that the dominant proportion of mass in fact comes from a \emph{single} peak which is supported by a ball of radius $\rad$. It is an important problem to further identify the profile of the mass distribution inside the localization region. The first step to tackle this problem is to understand what the environment looks like in the localization region, which is an interesting problem itself. These two problems are listed as the main questions in~\cite[Section~1.3]{Wolfgang16}. %The main results in this paper address them.

The first main result in this paper is about the behavior of environment in the localization region. Intuitively, the ball where the random walk will be localized should contain very few obstacles, or even no obstacle. It is proved in \cite{DX18} that the volume proportion of obstacles inside the localizing ball is at most $o(1)$, but it remains open to show that it actually contains no obstacle at all. Our first main result resolves this question.

\begin{thm}\label{main}
There exists a constant $\kappa>0$ depending only on $(d,p)$
% and $\mymathcal {x}_n = \mymathcal {x}_n(\ob)\in B(0,C(\log n)^{-2/d}n)$, the ball of radius $C(\log n)^{-2/d}n$ centered at the origin,
such that with $\widehat\P$-probability tending to one as $n\to\infty$,
\begin{equation}
\label{eq:thm-ballopen}
\mathcal B_n : =B(\mymathcal {x}_n,\rad - \rad^{1- \kappa}) \text{ is open.}
\end{equation}
%where $B(x,r)$ denotes the Euclidean ball with center $x$ and radius $r$.
% and
\end{thm}
\begin{remark}
\label{rem:annealed}
Under the \emph{annealed} law $\P\otimes\Pr(\cdot\mid \tau>n)$, a similar ball clearing phenomenon was proved for $d = 2$ in \cite{Bolthausen94} and~\cite{S91}, while the latter work studied a continuum analogue called Brownian motion among Poissonian obstacles. The extension to $d \geq 3$, conjectured in~\cite{Bolthausen94}, was open for a long time but recently resolved in~\cite{DFSX18} and~\cite{BC18} independently.
\end{remark}
The proof of Theorem~\ref{main} differs substantially from those of the annealed result mentioned in Remark~\ref{rem:annealed} and requires new ideas. It is based on intricate bounds on the tail distribution of the principal Dirichlet eigenvalue of the random walk in the island of localization, and relies on environment switching arguments that add or remove obstacles. The heart of the paper consists in making judicious choices of which obstacles to add or remove, and estimating how such modifications change the associated principal Dirichlet eigenvalue, which is of independent interest. See Section~\ref{sec:Proof Outline} for a more detailed proof outline.

Our second main result gives the limiting law of $\rad^{-1}(S_t - \mymathcal {x}_n)$ for $t=n$ or $t$ in the bulk, conditioned on survival up to time $n$.
The result for $t=n$ corresponds to the limiting profile of the solution of~\eqref{eq:PAM-killing}. Thanks to Theorem~\ref{main}, we are able prove the convergence at the level of local limit theorem.
Let $\phi_1$ and $\phi_2$ be the $L^1$ and $L^2$-normalized first eigenfunction of the Dirichlet-Laplacian of the unit ball in $\R^d$, respectively, and let $|\cdot|_1$ and $|\cdot|$ denote $\ell^1$ and $\ell^2$ norms on $\Z^d$, respectively.
   \begin{thm}   	\label{thm:distribution}
There exists $C>0$ depending only on $(d,p)$ such that the following hold with $\widehat\P$-probability tending to one as $n\to\infty$:
\begin{gather}
\label{eq:thm-ballloc}
	\min_{C |\mymathcal {x}_n| \leq t \leq n}\Pr( S_t \in \mathcal B_n \mid \tau>n)  \to 1\,,\\
\label{eq:thm-endpt}
	\sup_{x\in\mathcal B_n : |x|_1 + n \text{ is even}}\left|\rad^{d} \Pr(S_{n} = x \mid \tau>n) - 2\phi_1(\tfrac{x - \mymathcal {x}_n}{\rad})\right| \to 0\,,\\
\label{eq:thm-bulk}
	\sup_{x\in\mathcal B_n : |x|_1 + t \text{ is even}}\left| \rad^{d} \Pr(S_t = x \mid \tau>n) - 2\phi_2^2(\tfrac{x - \mymathcal {x}_n}{\rad})\right| \to 0\,,
\end{gather}
where $\mymathcal {x}_n=\mymathcal {x}_n(\ob)\in\Z^d$ is as in Theorem \ref{ballthm}, and the convergence in \eqref{eq:thm-bulk} holds uniformly for all $t\in [C |\mymathcal {x}_n|,  n - C  \rad^2 \log \rad]$.
   \end{thm}
Let us explain the heuristics behind Theorem~\ref{thm:distribution}. The endpoint result~\eqref{eq:thm-endpt} is rather simple. As we explained after Theorem~\ref{ballthm}, the endpoint profile should correspond to an eigenfunction supported on $B(\mymathcal {x}_n,\rad)$. Since Theorem~\ref{main} ensures that there is no obstacle in that ball, the eigenfunction looks like $\phi_1$. The ``midpoint'' result~\eqref{eq:thm-bulk} is better explained in probabilistic language. From Theorem~\ref{ballthm}, we expect that the random walk stays around $B(\mymathcal {x}_n,\rad)$ after reaching it. If we knew that it stays in the ball, then the random walk after reaching $B(\mymathcal {x}_n,\rad)$ would scale to a Brownian motion conditioned to stay in a ball, for which the midpoint distribution profile looks like $\phi_2^2$ (roughly speaking, one factor $\phi_2(\tfrac{x - \mymathcal {x}_n}{\rad})$ represents the probability of going to $x$ without exiting the ball, and the other $\phi_2(\tfrac{x - \mymathcal {x}_n}{\rad})$ represents the probability of starting from $x$ and not exiting the ball).

However, there is a caveat to the last part. We do not know whether the random walk under $\Pr(\cdot \mid \tau>n)$ stays inside $B(\mymathcal {x}_n,\rad)$ after reaching it, even if we make the radius slightly larger. Rather we expect the random walk to make excursions of lengths up to $c\log n$ away from $B(\mymathcal {x}_n,\rad)$. This is in sharp contrast to the annealed results~\cite{Bolthausen94,Pov99,S91}. One of the key fact in our proof is that for any $t$ chosen as in Theorem~\ref{thm:distribution}, with high probability, the random walk stays in $B(\mymathcal {x}_n,(1+\epsilon)\rad)$ for long enough time before and after $t$ so that the above scaling argument remains valid. 

%\begin{remark}
%The limiting marginal distribution identified by Theorem \ref{thm:distribution} is consistent with the behaivour of a random walk conditioned to stay inside the ball $\mathcal B_n$ (see \eqref{eq:thm-ballopen}) up to time $n$ after reaching it, although in our case we expect the walk to make excursions of lengths up to $c\log n$ away from $\mathcal B_n$.
%\end{remark}
%The proof of Theorem \ref{thm:distribution} is based on eigenfunction decompositions and eigenvalue and eigenfunction estimates. \red{add some references and more heuristics?}
\subsection{Related works}
\label{sec:literature}
Let us first review some known results for the {\em random walk among Bernoulli obstacles} in the quenched setting. This model has been studied intensively by Sznitman in the context of a space-time continuum analogue called \textit{Brownian motion among Poissonian obstacles}. He developed a coarse graining scheme called the \textit{method of enlargement of obstacles} to analyze this model.
As an early application, the logarithmic asymptotics of the survival probability was first derived in \cite{Sznitman93b}, which supported the picture of localization to a ball of radius $\rad$.
Later in \cite{Sznitman96}, Sznitman proved the ``pinning effect'' for Brownian motion. It states that, conditioned on survival up to time $n$, the Brownian motion will go to one of $n^{o(1)}$ many islands, each of diameter $n^{o(1)}$ and then stay there afterwards. In~\cite[Theorem~4.6]{Szn97b}, it was further proved that when $d=2$, there is a vacant ball of radius almost $\rad$ whose Dirichlet principal eigenvalue is very close to that of the macrobox $(-n,n)^d$. However, it remained unclear whether the Brownian motion gets localized in that vacant ball or not.
%, which is the content of our Theorem~\ref{main}.
More results for this model can be found in the monograph \cite{Sznitman98}.

Recently, significant improvements on the localization picture have been made in~\cite{DX17,DX18} for the discrete space-time model. The first paper~\cite{DX17} improved the bound for the number of islands to $(\log n)^C$ for some constant $C$, and also proved that the random walk hits the union of these islands rather quickly and then stays there. This latter result is a kind of path localization that had been known only for $d=1$. Then in the second paper~\cite{DX18}, it is finally proved that the random walk gets localized in a single island, which is called a ``one city theorem'' in the parabolic Anderson model literature. Our Theorems \ref{main} and \ref{thm:distribution} give a more precise picture of the localization established in \cite{DX17,DX18}.

Our model fits in the general framework of the parabolic Anderson model, where one places random potential $\omega(x)$ at each $x \in \Z^d$, and consider the following quenched Gibbs measure for the random walk path $S$:
\begin{equation}
\label{eq:PAM}
	\frac{\Ex\big[\exp(\sum_{i=1}^n\omega(S_i));S \in \cdot\big]}{\Ex\big[\exp(\sum_{i=1}^n\omega(S_i))\big]}\,.
\end{equation}
Our model corresponds to the case where $\omega$ is independent and identically distributed random variables taking value $0$ or $-\infty$.
The model~\eqref{eq:PAM} is well-studied for $\omega$ unbounded from above. Among other things, it has recently been proved in~\cite[Theorem~2.9]{BKdS18} that when the distribution of $\omega$ has a double-exponential tail, the random walk localizes in an island of size $O(1)$ and the endpoint distribution converges to the principal eigenfunction of a certain Schr\"odinger operator. This corresponds to our Theorem~\ref{thm:distribution}. For distributions with heavier tails, it has been proved that the random walk localizes at a single point~\cite{KLMS09,LM12,ST14,FM14}. For the distributions with lighter tails, the localization picture is incomplete. More precisely, the following two cases remain open:
\begin{itemize}
 \item the distribution is unbounded from above with a tail lighter than the double exponential. This class is called \emph{almost bounded} in this context and studied in~\cite{HKM06}.
 \item the distribution bounded from above with a support different from $\{0, -\infty\}$. This class contains the Brownian motion among soft Poisson obstacles studied in~\cite{Sznitman98} and a more general situation studied in~\cite{BK01}.
\end{itemize}
For an up-to-date review of the parabolic Anderson model, we refer the reader to the recent monograph by K\"{o}nig \cite{Wolfgang16}.

\subsection{Organization of the paper and notation}
The rest of the paper is organized as follows. We will first collect some results from \cite{DX17,DX18} in Section \ref{sec:Preliminaries}, which we will need later in the proof. Then in Section \ref{sec:Proof Outline}, we introduce some intermediate results and use them to prove Theorems \ref{main}, \ref{thm:distribution}. Results introduced in Section \ref{sec:Proof Outline} will then be proved in Section \ref{sec:Ball Clearing} and \ref{sec:localization}.

Throughout the rest of the paper, $C, c$ will denote  positive constants depending only on $(d, p)$ whose numerical values may vary from line to line with $C$ typically a large constant and $c$ a small constant. For constants such as $C_1$, $c_2$ or $\kappa$ (which also depend only on $(d, p)$), their values will stay the same throughout the paper. Unless stated otherwise, all asymptotic statements concerning the law of $\ob$ would apply with $\widehat \P$ probability tending to one as $n\to\infty$.

\section{Preliminaries}
\label{sec:Preliminaries}
We recall here some basic results and tools developed in \cite{DX18} that we will need in our proof. The main result of \cite{DX18} was that conditioned on quenched survival up to time $n$, the random walk will be confined in an island of diameter $(\log n)^C$ during the time interval $[Cn(\log n)^{-2/d},n]$.
Furthermore, with high probability the random walk $S_t$ is in a ball whose radius is asymptotically equal to $\rad$. More precisely, the following was proved in \cite{DX18}, where
\begin{equation}
 S_{[k,l]}=\{S_i\colon k\le i \le l\}
\end{equation}
denotes the range of the random walk during the time interval $[k,l]$.
\begin{thmx}[{\cite[Lemma 6.11 and Remark 6.2.]{DX18}}]
\label{thm:Ball confinement}There exist constants $\rf{c_1,C_1>0}$, $v_* =v_*(\rf{n},\ob)\in B(0,Cn(\log n)^{-2/d})$ and $$\ct{\fregion} \in \fregion:=\text{the connected component in $B(v_*,(\log n)^{C_1}) \setminus \ob$ that contains $v_*$ }$$
such that the following holds:
Let
\begin{equation}
	\hbn := B(\ct{\fregion},(1 + \rad^{-c_1})\rad)\setminus \ob.
\end{equation}
Then with $\widehat\P$-probability tending to one as $n \to \infty$, we have $\hbn \subset \fregion$ and
\begin{equation}
\label{eq:hit-ball-old}
	\Pr(\tau_{\hbn}<C|\ct{\fregion}|, S_{[\tau_{\hbn},n]} \subset \fregion \mid \tau > n) \geq 1-  \exp(-\rad^{c})\,;
\end{equation}
furthermore, uniformly in $ t \in [C|\ct{\fregion}|,n]$,
\begin{equation}
\label{eq:singball-ball}
		\Pr(S_t  \in \hbn \mid \tau>n) \geq 1- \exp(-\rad^c)\,.
\end{equation}
\end{thmx}
\begin{figure}
\begin{tikzpicture}
\draw [thin] (0,0) circle (2cm);

\foreach \i in {1,...,500}
{ \fill ({random*360}:{3cm*sqrt(random}) circle [radius=0.7pt]; } ;
 \draw [loosely dotted,fill=white,fill opacity = 1] (0.35,0.65) circle (0.7cm);
\node[fill,opacity=1,fill=white,rounded corners=2ex] at (0,0) { $v_*$};
\node at (0.35,0.65) { {$\ct{\fregion}$}};
\node[fill,opacity=1,fill=white,rounded corners=2ex] at (-1.4,-0.6) { {$\fregion$}};

\end{tikzpicture}
\centering
\caption{The centers $v_*$ and $\ct{\fregion}$ in Theorem \ref{thm:Ball confinement}. Little dots are obstacles. Lemma~\ref{Onecity-Ball} asserts that $\ct{\fregion}$ is the center of an almost vacant ball.}
\label{figure0}

\end{figure}
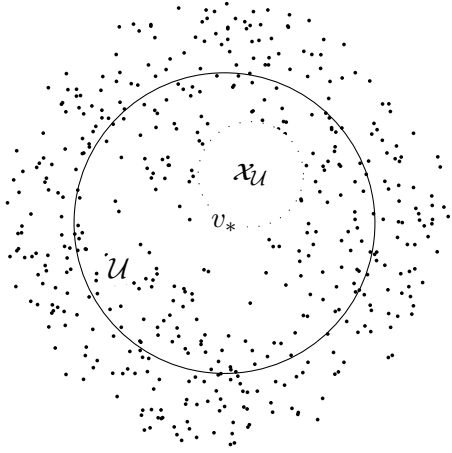

Theorem \ref{thm:Ball confinement} asserts that conditioned on survival up to time $n$, the walk reaches $\fregion$ with linear speed and then stays confined to $\fregion$ till time $n$ with high probability. Furthermore, at each $ t \in [C|\ct{\fregion}|,n]$, with high probability, the conditioned walk is localized in the ball $\hbn$ with center $\ct{\fregion}$, which is the same as $\mymathcal {x}_n$ in Theorem \ref{main}.

Let us briefly recall how $v_*$, and hence $\fregion$, is determined in the theorem above. Basically, there are small number of local regions like $\mathcal{U}$ which have an atypically large eigenvalue for the transition matrix (small survival cost) and are within a rather small distance from the origin (small crossing cost). The set $\mathcal{U}$ is then determined as the region which minimizes the sum of above two costs, see (3.5) in~\cite{DX18}.

We will need the following two properties for the set $\fregion$. The first one asserts that the eigenvalue for $\mathcal{U}$ is atypically large, which is expected from the above definition. The second asserts that there is a ball almost free of obstacles with radius $\rad$ in $\fregion$, which essentially follows from the first one and an quantitative isoperimetric inequality.
\medskip

\noindent \textbf{(a) The eigenvalue $\lambda_\fregion$.} For $A \subset \Z^d$, we let $\lambda_{A}$ denote the principal (largest) eigenvalue of $P|_{A}$, which is the transition matrix of the simple symmetric random walk  on $\Z^d$ killed upon exiting $A$. The following result gives a deterministic lower bound $\lambda_*$ on $\lambda_\fregion$ and shows that it is very close to $1$. In particular, there are at most $(\log n)^C$ many such islands in $B(0,n)$ with eigenvalues larger than $\lambda_*$.

\begin{lemma}
\label{lambda*}
There exists
 \begin{equation}
 \label{eq:lambda*lb}
 	\lbs  = \lambda_*(n,d,p)\geq 1 - \mb \rad^{-2} - C_* \rad^{-3}\,,
 \end{equation}
where $\mb$ is the first Dirichlet-eigenvalue of $-\tfrac{1}{2d}\Delta$ in the unit ball and $C_*$ is a constant depending only on $(d,p)$ such that \begin{equation}
\label{eq:lambdaU>*}
 	\lim_{n \to \infty}\widehat\P(\lambda_\fregion \geq \lambda_*) = 1\,.
 \end{equation}
Furthermore, if we denote
\begin{equation}
\label{eq:def-plr}
 	\plr = B(0,(\log n)^{C_1}) \setminus \ob\,,
 \end{equation}
 then for some constants $C,c>0$ and $n$ sufficiently large,
\begin{equation}
 \label{eq:lambda*d}
	n^{-d}(\log n)^{c} \leq \P(\lambda_\plr \geq \lambda_*) \leq n^{-d}(\log n)^{C}\,.
\end{equation}
\end{lemma}
\begin{proof}
This follows from results in \cite{DX17,DX18}. As in \cite{DX18}, we choose the cutoff $\lambda_* := p_{\alpha_1}^{1/k_n}$ where $k_n = (\log n)^{4-2/d}(\log\log n)^{2\11_{d=2}}$ and $p_{\alpha_1}$ (defined in \cite[(3.1)]{DX17}) is appropriately chosen according to some large quantile of the distribution of survival probability up to $k_n$ steps. Then \eqref{eq:lambdaU>*} can be found in \cite[Lemma 2.1]{DX18} and \eqref{eq:lambda*lb} can be found in \cite[Lemma 2.5]{DX18}. The lower bound in \eqref{eq:lambda*d} follows from \cite[(2.2)]{DX18}. The upper bound in \eqref{eq:lambda*d} can be proved using \cite[Lemma 3.3]{DX17} and adapting the proof of \cite[(3.4)]{DX17} by changing the value of $R$ there to $(\log n)^{C_1}$.
\end{proof}

\noindent \textbf{(b) An almost open ball in $\fregion$.} We want to find an open subset of $\fregion$ which is very close to a ball. This is accomplished by a coarse graining argument that first divides $\fregion$ into two types of mesoscopic boxes according to the local obstacles density. Intuitively, the random walk tends to stay in the low obstacle density region. The key ingredient in proving \eqref{eq:singball-ball} is that the low obstacle density region is very close to a ball. To elaborate on this, it is convenient to shift the center of localization to the origin and work with $\plr$ defined in \eqref{eq:def-plr}. Roughly speaking, $\fregion$ is the best among all possible translates of $\plr$ in $[-n,n]^d$.

Let $|\cdot|_\infty$ denote the $\ell^\infty$-norm on $\Z^d$ and denote the $\ell^\infty$ ball of radius $r$ (or box of side length $2r+1$) by
\begin{equation}
\label{eq:def-K-box}
	K(v,r) := \{x \in \Z^d : |x-v|_\infty \leq r\}\,.
\end{equation}
For $\epsilon \in (0,1)$, which may depend on $n$, we consider the following disjoint boxes that cover $\Z^d$: $$K(x,\lfloor \epsilon \rad \rfloor) \ for  \ x \in  (2\lfloor \epsilon \rad\rfloor+1)\Z^d\,.$$

\begin{defn}
\label{def:empty}
For $\epsilon \in (0,1)$, let $\emp_n(\epsilon)$ be the union of boxes $K(x,\lfloor \epsilon \rad\rfloor)$ that intersect $\plr$ (defined in \eqref{eq:def-plr}) such that  $|\ob \cap K(x,\lfloor \epsilon \rad\rfloor)| \leq \epsilon |K(x,\lfloor \epsilon \rad\rfloor)|$, where $|K(x,\lfloor \epsilon \rad\rfloor)|$ denotes the cardinality of the set $K(x,\lfloor \epsilon \rad\rfloor)$.
\end{defn}
By the same combinatorial calculation as in the proof of \cite[Lemma 5.2]{DX18}, one can show that typically the volume of the low obstacle density region $\emp_n(\epsilon)$ is at most $C \rad^d$. More precisely, we have the following lemma.

\begin{lemma}
\label{empty}
There exists a constant $c \in (0,1)$ such that for any $\epsilon \in( \rad^{-1/2},c)$,%$-\alpha \log \alpha\leq p/5$,
\begin{equation*}
%|\xmp(\epl,\alpha)| \leq \omega_d(2\epl)^{-d} +\epl^{-d}\alpha^{1/3} \AND
\P(|\emp_n(\epsilon)| \geq |B(0,\rad)|+ \epsilon^{1/2}\rad^{d}) \leq e^{-\rad}n^{-d}\,.
\end{equation*}
Also, there exists $\rf{C_2>0}$ such that
\begin{equation*}
%|\xmp(\epl,\alpha)| \leq \omega_d(2\epl)^{-d} +\epl^{-d}\alpha^{1/3} \AND
\P(|\emp_n(\epm)| \geq C_2 \rad^d) \leq n^{-100d}\,.
\end{equation*}
\end{lemma}
The following lemma is one of the key ingredients in proving Theorem \ref{thm:Ball confinement}. It says that if $\lambda_\plr$ is grater than or equal to $\lambda_*$ (which is close to $\lambda_{B(0,\rad)}$, see \eqref{eq:lambda*lb}), then typically $\emp_n(\epsilon)$ is very close to a ball of radius $\rad$.
\begin{lemma}[{\cite[Lemmas 5.2--5.9]{DX18}}]
\label{Onecity-Ball}
Let $\lambda_*$ be as in Lemma \ref{lambda*}. There exists $c_2>0$ sufficiently small such that if we denote\begin{equation}
\label{eq:def-epn}
	\epn := \rad^{-c_2} \,,
\end{equation} and assume that $ \lambda_{\plr} \geq \btn $ and
\begin{align}
|\emp_n(\epsilon)| \leq |B(0,\rad)| + \epsilon^{1/2} \rad^d, \text{ for } \epsilon = \varepsilon^{1/2}_n,\epn,\varepsilon^{2}_n\,,\label{eq:empty prob E}
\end{align}
then there exists $\ct{\plr} \in \plr$ such that
\begin{equation}
\label{eq: Ball and E}
	|B(\ct{\plr},\rad) \triangle \emp_n(\epsilon)| \leq \epn^{1/4} \rad^d\,,
\end{equation}
where $\triangle$ stands for the symmetric difference of two sets.
\end{lemma}

It follows immediately that the volume proportion  of obstacles in $B(\ct{\plr},\rad)$ is very small:
\begin{equation}
\label{eq:obrad}
	|B(\ct{\plr},\rad) \cap \ob|  \leq |B(\ct{\plr},\rad) \triangle \emp_n(\epsilon)| + \epn |\emp_n(\epsilon)| \leq C\epn^{1/4} \rad^d\,,
\end{equation}
which is the starting point of our proof. In particular, since $\fregion \subset B(0,n)$, we can deduce from Lemma \ref{empty} that with high probability, $\fregion$ satisfies the same volume control as in \eqref{eq:empty prob E}, and hence the volume proportion of obstacles in $B(\ct{\fregion},\rad)$ is very small.

\section{Proof Outline}
\label{sec:Proof Outline}
In this section, we list the key intermediate results and prove Theorem~\ref{main} and Theorem~\ref{thm:distribution} assuming those results.

\subsection{Ball Clearing}
By \eqref{eq:lambda*d}, we know that there are at most $(\log n)^{C}$ many balls of radius $(\log n)^{C_1}$ in $B(0,n)$ with eigenvalues at least $\lambda_*$. The fact that $B(\ct{\fregion},\rad - \rad^{1-\kappa})$ is open with high probability will then follow from the following proposition.
\begin{proposition}
\label{po-clear}Let $\plr$ be as in \eqref{eq:def-plr} and let $\lambda_*$ be as in Lemma \ref{lambda*}. Then there exist $\kappa \in (0,1)$ and $C>0$ such that for all $n$ sufficiently large,
		\begin{equation}
		\label{eq:po-clear}
			 \P( B(\ct{\plr},\rad - \rad^{1-\kappa}) \cap \ob \not = \varnothing \mid \lambda_\plr \geq \lambda_*)
			 \leq Ce^{-\rad^{1/3}} \,.
	\end{equation}
\end{proposition}
\begin{proof}[\bf Proof of Theorem \ref{main}]
In view of Theorem~\ref{thm:Ball confinement}, it suffices to prove \eqref{eq:thm-ballopen} with $\mymathcal x_n$ replaced by $\ct{\fregion}$.
Let us denote the translate of $\plr$ by
\begin{equation*}
  	\plr(x) = B(x,(\log n)^{C_1}) \setminus \ob\,.
\end{equation*}
By Proposition \ref{po-clear} and \eqref{eq:lambda*d},
$$ \P(\lambda_{\plr} \geq \lambda_*, B(\ct{\plr},\rad - \rad^{1-\kappa})\text{ is not open}) \leq Ce^{-\rad^{1/3}} (\log n)^C n^{-d} = o(n^{-d})\,.$$
This yields that with $\P$-probability tending to one as $n \to \infty$, for all $x \in B(0,n)$,
$$ \text{either } \lambda_{\plr(x)} < \lambda_* \text{ or }B(\ct{\plr(x)},\rad - \rad^{1-\kappa})\text{ is open}\,.$$
Recall that we proved in Lemma \ref{lambda*} that $\lambda_\fregion \geq \lambda_*$. Hence, $B(\ct{\fregion},\rad - \rad^{1-\kappa})$ is open with $\widehat\P$-probability tending to one as $n\to\infty$.
\end{proof}

The proof of Proposition \ref{po-clear} is based on the following heuristics.  Suppose $B(\ct{\plr},\rad - \rad^{1-\kappa})$ is not completely open, then we consider the operation that removes all obstacles inside $B(\ct{\plr},\rad - \rad^{1-\kappa})$. After performing such an operation, the eigenvalue $\lambda_\plr$ will increase, for example, from $\lambda_\plr$ to $\lambda_\plr +\delta$. Such an operation will yield the following inequality:
\begin{equation}
\label{eq:po-withob}
	\P(B(\ct{\plr},\rad - \rad^{1-\kappa}) \cap \ob \not = \varnothing,\lambda_\plr>\lambda_*) \leq C(n,\delta,d,p) \P(\lambda_\plr>\lambda_* +\delta)\,.
\end{equation}
Now if the tail of the probability distribution of $\lambda_\plr$ is not very heavy in the sense that
\begin{equation}
\label{eq:po-tail}
	\frac{\P(\lambda_\plr>\lambda_* +\delta)}{\P(\lambda_\plr>\lambda_* )} \ll 1\,,
\end{equation}
and the factor $C(n,\delta,d,p)$ in \eqref{eq:po-withob} is small compare to \eqref{eq:po-tail}, then we have
\begin{equation*}
	\P(B(\ct{\plr},\rad - \rad^{1-\kappa}) \cap \ob \not = \varnothing \mid \lambda_\plr>\lambda_*) \leq C(n,\delta,d,p) \frac{\P(\lambda_\plr>\lambda_* +\delta)}{\P(\lambda_\plr>\lambda_* )} \ll 1\,,
\end{equation*}
which yields Proposition \ref{po-clear}. Therefore it suffices to establish in a more precise manner the two ingredients \eqref{eq:po-withob} and \eqref{eq:po-tail}.
\begin{remark}
In \cite[Proposition 2.2]{DFSX18}, we have proved an analogue of Proposition \ref{po-clear} under the annealed polymer measure, where we used operations that modify the obstacle configurations and the random walk paths jointly. The difficulty in the quenched setting is that we need to identify a vacant ball in $\fregion$, for which we only know $\lambda_{\fregion} \ge \lambda_*$ from Lemma~\ref{lambda*}. This is why Proposition~\ref{po-clear} is formulated in terms of the eigenvalue and as a result, we can perform operations only on the obstacle configurations. Nevertheless, it is worth mentioning that operations that modify obstacle configurations or random walk paths play an important role both in this paper and in \cite{DFSX18}.
\end{remark}

The following result makes \eqref{eq:po-tail} precise and shows that the tail of the probability distribution of $\lambda_\plr$ is not too heavy.

\begin{lemma}
\label{eigenvalue tail - 1}
Let $\consa \geq 1$. There exists a constant $c_\consa>0$ depending only on $(\consa,d,p)$ such that for all $n$ sufficiently large, and for all $\beta \geq 1 - \consa \rad^{-2}$ and
	\begin{equation}
	\label{eq:e-range-1}
		\emm \in \big((\log\log n)^4\rad^{-d},c_\consa \big)\,,
	\end{equation}
	 we have
		\begin{equation}
		\P(\lambda_\plr \geq \beta) \leq e^{-\emm(\log(1/\epsilon))^{-3} \rad^d}\P(\lambda_\plr \geq \beta - \epsilon \rad^{-2}) + n^{-10d}\,.
	\end{equation}
\end{lemma}

One of the challenges in proving results of the type \eqref{eq:po-withob} is that how much $\lambda_\plr$ increases after removing the obstacles depends on the local configuration around the obstacles. For example, if the obstacles being removed are near the boundary of $B(\ct{\plr}, \rad)$ where there are a lot of unremoved obstacles, then the effect of the removal would be very small (more discussions about this issue can be found at the beginning of Section \ref{sec:tail2}). To quantify the effect of removing certain obstacles, we suppose $\lambda_\plr \geq \btn$ and \eqref{eq:empty prob E} holds. For $\delta \in(0,1/2)$, which may depend on n,  and nonnegative integer $k$, define
\begin{equation}
\label{eq:def-b-dk}
	B_{\delta,k} = B(\ct{\plr},(1 - \delta + 2^{-k}\delta)\rad)\,.
\end{equation}
Then for all $k\geq0$,
$B(\ct{\plr},(1 - \delta)\rad) \subset B_{\delta,k+1}\subset B_{\delta,k} \subset B(\ct{\plr},\rad)\,.$ For any $\delta>0$, if $B(\ct{\plr},(1 - \delta)\rad)$ is not completely open, then we define (for some constant $c_5 \in (0,1)$ to be chosen in Lemma \ref{RemoveOb})
\begin{equation}
\label{eq:def-J}
\lj = \lj_{\delta} :=  \min\{k \in \N:|\ob \cap B_{\delta,k}| \geq c_5|\ob \cap B_{\delta,k-1}|\} \,,
\end{equation}
which must be finite due to the assumption that $B(\ct{\plr},\rad - \rad^{1-\kappa})$ is not completely open (see Figure \ref{figure1}).

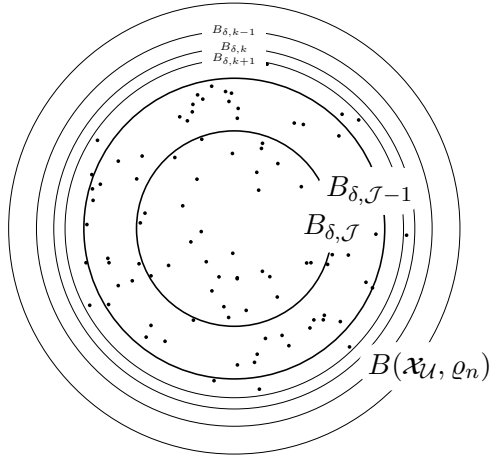
\begin{figure}
\begin{tikzpicture}
\draw [ultra thin] (0,0) circle (3cm);
\draw [ultra thin] (0,0) circle (2.25cm);
\draw [ultra thin] (0,0) circle (2.4cm);
\draw [ultra thin] (0,0) circle (2.63cm);
\draw [semithick] (0,0) circle (2cm);
\draw [semithick] (0,0) circle (1.3cm);
\foreach \i in {1,...,100}
{ \fill ({random*360}:{2.3cm*sqrt(random}) circle [radius=0.7pt]; } ;
     \node[fill,opacity=1,fill=white,rounded corners=2ex] at (1.3,0) {$B_{\delta, \mathcal J}$};
\node[fill,opacity=1,fill=white,rounded corners=2ex] at (1.8,0.5) {$B_{\delta, \mathcal J-1}$};
\node[fill,opacity=1,fill=white,rounded corners=2ex] at (2.6,-1.832) {$B(\ct{\fregion},\rad)$};
\node[fill,opacity=1,fill=white,rotate = 90] at (0,2.5) {......};
\node[fill,opacity=1,fill=white] at (0,2.63) {\tiny  \scalebox{.7} {$B_{\delta,k-1}$}};
\node[fill,opacity=1,fill=white] at (0,2.25) {\tiny  \scalebox{.7} {$B_{\delta,k+1}$}};
\node at (0,2.4) { \tiny \scalebox{.7} {$B_{\delta,k}$}};

\end{tikzpicture}
\centering
\caption{The balls $B_{\delta,k}$ and $B_{\delta,\lj}$ as defined in \eqref{eq:def-b-dk} and \eqref{eq:def-J}. Little dots are obstacles.}
\label{figure1}

\end{figure}

The following result makes \eqref{eq:po-withob} more precise and says that removing $m$ obstacles in $B_{\delta,\lj-1}$ will increase the eigenvalue $\lambda_\plr$ by $(m/\rad^d )^{1 - 1/d} \rad^{-2}$.

\begin{lemma}
\label{eigenvalue tail - 2 - not open}
Let $C_1$ be defined as in Theorem \ref{thm:Ball confinement} and \rf{$\delta = \rad^{-\kappa}$}. There exist constants $\kappa \in (0,1)$ \rf{and} $C >0$ such that \rf{the following holds} for all $n$ sufficiently large: \rf{For} any $1 \leq m \leq \rad^d$ \rf{and} $\beta \geq \btn$,
	\begin{equation}
	\label{eq:eigenvalue tail - 2 - not open}
		\P(\lambda_\plr \geq \beta , |\ob \cap B_{\delta,\lj-1}| = m,\eqref{eq:empty prob E}) \leq C \rad^{dC_1} \Big(\frac{C\rad^d}{m}\Big)^{m}\P(\lambda_\plr \geq \beta +  (m/\rad^d )^{1 - 1/d} \rad^{-2})	\,.
	\end{equation}
\end{lemma}
 Both Lemmas \ref{eigenvalue tail - 1} and \ref{eigenvalue tail - 2 - not open} are estimates on the tail distribution of $\lambda_\plr$, but in opposite directions. The common strategy in both proofs is obstacle modification. To prove Lemma \ref{eigenvalue tail - 1}, we judiciously add obstacles and show that we get a large gain in probability for the obstacle configuration but little decrease in $\lambda_\plr$. To prove Lemma \ref{eigenvalue tail - 2 - not open}, we judiciously remove obstacles and show that we get a large gain in $\lambda_\plr$ while the probability of the obstacle configuration changes little. We will prove Proposition \ref{po-clear}, Lemmas \ref{eigenvalue tail - 1} and \ref{eigenvalue tail - 2 - not open} in Section \ref{sec:Ball Clearing}.

\subsection{Random Walk Localization}

We know from Theorem \ref{thm:Ball confinement} that conditioned on survival up to time $n$, the random walk stays in $\fregion$ during the time interval $[C|\ct{\fregion}|,n]$ with high probability. To give a more detailed description for the random walk in this time window, it is convenient to consider the random walk conditioned to stay in $\fregion$ for a long time.

Recall from Theorem \ref{thm:Ball confinement} that $\hbn := B(\ct{\fregion},(1 + \rad^{-c_1})\rad) \setminus \ob$ is a ball of radius slightly larger than $\rad$. Knowing that $B(\ct{\fregion},\rad - \rad^{1-\kappa})$ is open with $\widehat \P$-probability tending to one as $n \to \infty$, it is straightforward to deduce the following result from \cite[(6.15)]{DX18}. The details will be given in Section \ref{sec:transition probability}.
\begin{lemma}
\label{pt-loc} There exist constants $C_3,c>0$ such that for any $\varepsilon>0$, the following holds with $\widehat\P$-probability tending to one as $n \to \infty$: For any $z$ and $t$ satisfying 
\begin{equation}
\label{eq:ass-pt-loc}
	\text{either } z \in B(\ct{\fregion},(1 - \varepsilon)\rad)\cap\fregion \text{ and }t \geq 0, \quad \text{   or   }\quad z \in \fregion\text{ and } t \geq \rad^{C_3}\,,
\end{equation}
and for all $m\ge t$, 
\begin{equation}
\label{eq:pt-loc}
	\Pr^z(S_{t} \not \in \hbn\mid \tau_{\fregion^c}>m) \leq \exp(-\rad^c).
\end{equation}
\end{lemma}
We will strengthen this result in two steps :
\begin{itemize}
	\item[(a).] We will first prove that conditioned on $\{\tau_{\fregion^c}>m\}$, the probability that the random walk is at some site $x$ at a fixed time $\rad^2 \leq t \leq m$ is $O(\rad^{-d})$ uniformly in $x$.
	% In a ball of radius slightly smaller than $\rad$, we also have a uniform lower bound of order $\rad^{-d}$.
	Then combining with \eqref{eq:pt-loc}, we deduce that conditioned on $\{\tau_{\fregion^c}>m\}$, at any fixed time, with high probability the random walk will be localized in a ball of radius slightly smaller than $\rad$.
	\item[(b).]  We will then derive the limiting marginal distribution of the random walk at the end point and at a deterministic time in the bulk, conditioned on $\{\tau_{\fregion^c}>m\}$.
\end{itemize}

First we show that conditioned on $\{\tau_{\fregion^c}>m\}$, the random walk will hit the the deep interior of the ball $B(\ct{\fregion},\rad)$ within $(\log n)^C$ steps. This allows us to focus on the random walk starting from the deep interior of the ball.
\begin{lemma}
\label{hit ball}
There exist $b_2 \in (0,1)$ and $\rf{C_4,c >0}$ ($C_4$ to be defined in Lemma \ref{hit tOmega-as}) such that with $\widehat\P$-probability tending to one as $n \to \infty$, for all $m \geq  \rad^{C_4}$ and $u \in \fregion$,
\begin{equation}
\label{eq:hitball}
	\Pr^u(\tau_{B(\ct{\fregion},b_2\rad)} \geq  \rad^{C_4}  \mid \tau_{\fregion^c} > m) \leq \exp(-\rad^c)\,.
\end{equation}
\end{lemma}

The first improvement upon \eqref{eq:pt-loc} in Lemma \ref{pt-loc} (see $(a)$ after Lemma \ref{pt-loc}) is the following local limit result.
\begin{lemma}\label{loclemma2}
	Let $b_2$ be as in Lemma \ref{hit ball} and $ \epsilon \in (0,1)$. Then with $\widehat\P$-probability tending to one as $n \to \infty$, the following holds\rf{: For all $u \in B(\ct{\fregion},b_2\rad)$, $m \geq t \geq \rad^2$,} $y \in \fregion$ \rf{and} $x \in B(\ct{\fregion},(1- \epsilon)\rad)$ such that $|x-u|_1 + t$ is even:
\begin{align}
\label{eq:loclemma2-2}
	\Pr^{u}(S_t = y \mid \tau_{\fregion^c} > m) &\leq C \rad^{-d}\,,\\
\label{eq:loclemma2-1}
	\Pr^{u}(S_t = x \mid \tau_{\fregion^c} > m) &\geq c \epsilon^2\rad^{-d}\,.
\end{align}
\end{lemma}
\begin{remark}
The second assertion~\eqref{eq:loclemma2-1} will not be used in the proof of the main results. We include it to complement~\eqref{eq:loclemma2-2} and as a precursor to Theorem~\ref{thm:distribution}.
\end{remark}

Combined with  Lemma \ref{pt-loc}, the preceding lemma can then be used to show that the random walk at any fixed time $t$ will be localized in the ball centered at $\ct{\fregion}$ with radius $\rad\big(1 - o(1)\big)$. More precisely,
\begin{cor}
\label{loclemma}
Let $\kappa>0$ be defined as in Proposition \ref{po-clear}. There exists a constant $c>0$ such that with $\widehat\P$-probability tending to one as $n \to \infty$, for all $u \in B(\ct{\fregion},b_2\rad)$ and $m \geq t \geq 0$,
\begin{equation}
\label{eq:99999}
	\Pr^{u}(S_t \in  B(\ct{\fregion},(1 - 2\rad^{-\kappa})\rad) \mid \tau_{\fregion^c} > m)  \geq 1- \rad^{-c}\,.
\end{equation}
\end{cor}
Lastly, Theorem \ref{thm:distribution} will be proved using Corollary \ref{loclemma} and the following lemma, which says that conditioned on the random walk staying in $\qbn$ for sufficiently long time, the distribution of the random walk at the end point (or at a deterministic time in the bulk) will converge in total variation distance to the normalized first eigenfunction (or normalized eigenfunction squared) on $\qbn$.
\begin{lemma}
\label{distribution}
There exist constants $c,C_5>0$ such that the following holds with $\widehat\P$-probability tending to one as $n \to \infty$\rf{: for all $v,y \in B(\ct{\fregion},(1 - 2\rad^{-\kappa}) \rad)$ and $m,t \geq C_5 \rad^2\log\log n $ with $|y-v|_1 + m + t $ even,}
\begin{align}
\label{eq:endpt-m-t}
	&\sup_{x\in\mathcal B_n : |x-v|_1 + m \text{ is even}}\left|\rad^d\Pr^v(S_{m} = x \mid \tau_{\qbn^c}>m) - 2\phi_1(\tfrac{x - \ct{\fregion}}{\rad})\right| \leq \rad^{-c}\,,\\
\label{eq:bulk-m-t}
	&\sup_{x\in\mathcal B_n : |x-v|_1 + m \text{ is even}}\left| \rad^d\Pr^v(S_{m} = x \mid S_{m+t} =y,  \tau_{\qbn^c}>m + t) - 2\phi_2^2(\tfrac{x - \ct{\fregion}}{\rad})\right| \leq \rad^{-c}\,,
\end{align}
where $\phi_1$ and $\phi_2$ are respectively the $L^1$ and $L^2$-normalized first eigenfunction of the Dirichlet-Laplacian of the unit ball in $\R^d$.
\end{lemma}
Let us prove Theorem \ref{thm:distribution} assuming the above lemmas.

\begin{proof}[\bf Proof of Theorem \ref{thm:distribution}]
To prove \eqref{eq:thm-ballloc}, we first show that combining \eqref{eq:hit-ball-old} with Lemma \ref{hit ball} yields
\begin{equation}
\label{eq:hit-b2-n}
	\Pr(\tau_{B(\ct{\fregion},b_2\rad)}<C|\ct{\fregion}|, S_{[\tau_{\hbn},n]} \subset \fregion \mid \tau > n) \geq 1- \exp(-\rad^{c})\,.
\end{equation}
Indeed, by the strong Markov property at time $\tau_{\hbn}$,
	\begin{align*}
		&\Pr(\tau_{B(\ct{\fregion}, b_2 \rad)}> \tau_{\hbn} + \lceil \rad^{C_4} \rceil, \tau_{\hbn}<C|\ct{\fregion}|, S_{[\tau_{\hbn},n]} \subset \fregion, \tau> n)\\
		& = \Ex\big[\11_{\tau > \tau_{\hbn },\tau_{\hbn}<C|\ct{\fregion}|}\Pr^{S_{\tau_\hbn}}(\tau_{B(\ct{\fregion}, b_2 \rad)}>\lceil \rad^{C_4} \rceil, \tau_{\fregion^c} > n - \tau_\hbn)\big]\,.
	\end{align*}
Since $|\ct{\fregion}| \leq C n (\log n)^{-2/d}$ implies $n - \tau_\hbn \geq n/2$, it follows from Lemma \ref{hit ball} that the above quantity can be bounded from above by
\begin{align*}
	& \exp(-\rad^c)\Ex\big[\11_{\tau > \tau_{\hbn },\tau_{\hbn}<C|\ct{\fregion}|}\Pr^{S_{\tau_\hbn}}( \tau_{\fregion^c} > n - \tau_\hbn)\big]\\
	 &\quad= \exp(-\rad^c)\Pr(\tau > \tau_{\hbn },\tau_{\hbn}<C|\ct{\fregion}|,S_{[\tau_{\hbn},n]} \subset \fregion)\,.
\end{align*}
Combined with \eqref{eq:hit-ball-old}, this proves \eqref{eq:hit-b2-n}.

Now, let $T$ denote the hitting time of the ball $B(\ct{\fregion},b_2 \rad)$ to lighten the notation. We consider deterministic time $t$ with $C|\ct{\fregion}| \leq t \leq n$ and $x \in B(v,(1 - \epsilon)\rad)$ with $|x|_1 +t$ even. By the strong Markov property,
\begin{align*}
&	\Pr(S_{[T,n]} \subset \fregion, S_t \in B(\ct{\fregion}, \rad),t > T, \tau > n)\\
&\quad = \Ex \big[ \11_{\tau\wedge t > T} \Pr^{S_T}(S_{t-T} \in B(\ct{\fregion}, \rad), \tau_{\fregion^c} > n - T)\big]\,.
\end{align*}
By Corollary \ref{loclemma}, this equals $\Pr(S_{[T,n]} \subset \fregion, \tau > n,t > T)(1 + o(1))$. Combining this with \eqref{eq:hit-b2-n} gives \eqref{eq:thm-ballloc}.

Next, we turn to the proof of~\eqref{eq:thm-endpt} and~\eqref{eq:thm-bulk}. The basic idea is to restrict the walk to a time interval $[t_1, t_2]$ such that the walk does not exit $\hbn$ during this time interval, which then allows us to apply Lemma \ref{distribution}. To this end, we denote for $0<t_1<t_2$,
$$A_{t_1,t_2}:= \{ S_{t_1},S_{t_2} \in B(\ct{\fregion} ,(1 - 2\rad^{-\kappa})\rad),S_{[t_1,t_2]} \subseteq \qbn, S_{[C|\ct{\fregion}|, n]} \subseteq \fregion,\tau > n\}\,.$$
We first notice that for any $[t_1,t_2] \subseteq [C|\ct{\fregion}|, n]$ with $t_2 - t_1 \leq e^{\rad^c}$, combining \eqref{eq:hit-b2-n}, \eqref{eq:99999}, and a union bound for the event in \eqref{eq:pt-loc} over all $t \in [t_2, t_1]$ yields that
\begin{equation}
\label{eq:120301}
	\Pr(A_{t_1,t_2}\mid  \tau > n) \geq 1 - \rad^{-c}\,.
\end{equation}

To prove \eqref{eq:thm-bulk}, we choose $t_1 = t - \rad^3 ,t_2 = t + \rad^3 $, and denote
\begin{align*}
	 I_{v,w} &= \Pr(S_{t_1}=v,S_{[C|\ct{\fregion}|, t_1]} \subseteq \fregion,\tau > t_1) \cdot \Pr^v( S_{t_2 -t_1 }= w, \tau_{\qbn^c} > t_2-t_1)\cdot \Pr^w(\tau_{\fregion^c} > n-t_2)\,.
\end{align*}We have
\begin{align*}
	\Pr&(S_t = x, A_{t_1,t_2})
	 = \sum_{v,w \in B(\ct{\fregion} ,(1 - 2\rad^{-\kappa})\rad)}I_{v,w} \cdot \Pr^v( S_{t-t_1} = x\mid S_{t_2 -t_1 }= w, \tau_{\qbn^c} > t_2-t_1)\,.	
\end{align*}
Therefore
\begin{align*}
	&\sum_{x : |x|_1 + t \text{ is even}}\left|\Pr(S_t = x,  A_{t_1,t_2})  - 2\rad^{-d}\phi_2^2(\tfrac{x - \ct{\fregion}}{\rad})\Pr(  A_{t_1,t_2})\right| \leq  \sum_{v,w \in B(\ct{\fregion} ,(1 - 2\rad^{-\kappa})\rad)} I_{v,w} \\
	& \quad   \quad \quad \quad \times \sum_{x : |x|_1 + t \text{ is even}}\left|\Pr^v( S_{t-t_1} = x\mid S_{t_2 -t_1 }= w, \tau_{\qbn^c} > t_2-t_1)- 2\rad^{-d}\phi_2^2(\tfrac{x - \ct{\fregion}}{\rad})\right| \\
	& \quad \quad \quad \quad \quad \quad \leq \rad^{-c} \Pr(A_{t_1,t_2})\,.	
\end{align*}
where in the last step, we used \eqref{eq:bulk-m-t}. Combining this with \eqref{eq:120301} yields \eqref{eq:thm-bulk}.

Finally, choose $t_1 = n - \rad^3 ,t_2 = n$ and combining \eqref{eq:120301} with \eqref{eq:endpt-m-t} yields \eqref{eq:thm-endpt}.
\end{proof}

Lemmas \ref{pt-loc}, \ref{hit ball}, and \ref{loclemma2}, Corollary \ref{loclemma} will be proved in Section \ref{sec:transition probability}, and Lemma \ref{distribution} will be proved in Section \ref{sec:distribution}.

\section{Ball Clearing}
\label{sec:Ball Clearing}
In this section, we will first prove Lemmas \ref{eigenvalue tail - 1} and \ref{eigenvalue tail - 2 - not open} in Sections \ref{sec:tail1} and \ref{sec:tail2}, respectively, and then conclude the proof of Proposition \ref{po-clear} in Section \ref{sec:4.1}.
\subsection{Proof of Lemma \ref{eigenvalue tail - 1}}
\label{sec:tail1}

\subsubsection{Proof outline}In this section, we outline the proof of Lemma \ref{eigenvalue tail - 1}, which shows that the tail of the distribution of $\lambda_\plr$ is not too heavy. The basic strategy is obstacle modification.

Let $\ell \in \N$ and partition $\Z^d$ into disjoint boxes $ K(x,\lm)$ of side length $2\lm+1$ (see \eqref{eq:def-K-box}) for $ x \in (2\lm+1)\Z^d$.
Let $\plr$ be as defined in \eqref{eq:def-plr}.
\begin{defn}
\label{def:truly open}
A box $K(x,\lm)$ is said to be {\em \trueop\ }if
\begin{equation}
\label{eq:truly open}
	\max_{u \in K(x,\lm)}\Pr^u\big(S_{[0,\lm^2]} \subset( K(x,4\ell) \backslash \ob)\big) \geq 1/10\,.
\end{equation}
Let $C_1$ be as in Theorem \ref{thm:Ball confinement}. We fix $\lm$ and let
$\tos=\tos(\lm, n)$ denote the union of all \trueop\ boxes that intersect with $B(0,(\log n)^{C_1})$. 
\end{defn}

We first note that \trueop\ boxes are very rare:
\begin{lemma}
\label{toprob}There exists $c>0$ such that for all $\ell$ sufficiently large,
	\begin{equation}
	\label{eq:toprob}
	 	\P\big(K(x,\lm) \text{ is \trueop} \big) \leq \exp(-c\lm^d)\,.
	 \end{equation}
\end{lemma}
\begin{proof}
To prove \eqref{eq:toprob}, it suffices to show that for all $u \in K(x,\lm)$,
\begin{equation}
	\P\big(\Pr^u(S_{[0,\lm^2]} \cap \ob = \varnothing) \geq 1/10\big) \leq \exp(-c\lm^d)\,,
\end{equation}
which can be found in \cite[Definition 2.4, Lemmas 2.8 and 2.9]{DX17}.
\end{proof}
In light of Lemma \ref{toprob}, by {\em closing} a \trueop\ box, namely, changing the obstacle configuration in a \trueop\ box to typical configurations, we could gain much probability for the obstacle configurations.

On the other hand, we have the following result, which says that in a typical environment, we can find a \trueop\ box such that $\lambda_\plr$ will only decrease slightly after closing this \trueop\ box.
\begin{lemma}
\label{iter}
Fix $\ell \geq 1$. Let $\rf{C_6>1}$ be a constant to be chosen in Lemma \ref{typical}, and let $\consa_1>0,\consa_2 \in (0,1)$ be two arbitrary constants. Let $\epn$, $\emp_n(\epn)$, and $C_2>1$ be as defined in \eqref{eq:def-epn}, Definition \ref{def:empty}, and Lemma \ref{empty}, respectively. We assume
	\begin{equation}
	\label{eq:typical}
		\min_{x \in B(0,(\log n)^{C_1})}|B(x,C_6 \rad) \setminus \tos| \geq\rad^{d}, \quad |\emp_n(\epn)| \leq C_2 \rad^d\,,
	\end{equation}
and $\lambda_\plr \geq 1 - \consa_1 \rad^{-2}$.
Then for each $z \in  B(0,(\log n)^{C_1})$ such that $|\tos \cap B(z,C_6 \rad)| \geq (\consa_2\rad)^d$, there exists a \trueop\ box $K(x,\lm)$ with $x \in B(z,20\rad)$ such that
\begin{equation}
\label{eq:itereq}
	\lambda_{\plr \setminus K(x,10\lm)} \geq \lambda_\plr - C \consa_1^2\consa_2^{-2(d-1)}\lm^{2(d+2)}\rad^{-d-2}\,.
\end{equation}
where \rf{$C>0$ is a constant depending only on $(d,p)$.}
\end{lemma}

Lemma \ref{iter} is the key ingredient in the proof of Lemma \ref{eigenvalue tail - 1}. We will use Lemma \ref{iter} repeatedly (see Lemma \ref{Remove Truly open set} below) to show that we can find a number of \trueop\ boxes such that $\lambda_\plr$ will not be decreased much after closing them. The operation of changing these \trueop\ boxes to typical configurations will map the event $\{\lambda_\plr \geq \beta\}$ to $\{\lambda_\plr \geq \beta -\delta\}$, where $\delta$ depends on $\ell$ and the number of \trueop\ boxes being closed. Combining with Lemma \ref{toprob} will then give an upper bound for $\P(\lambda_\plr \geq \beta)/\P(\lambda_\plr \geq \beta -\delta)$ (see Lemma \ref{mllemma}.) The proof of Lemmas \ref{iter}, \ref{Remove Truly open set} and \ref{mllemma} will be provided in Section \ref{sec:proof of iter}.
In Section \ref{sec:tail-1}, we fix appropriate choices of $\lm$ and the number of \trueop\ box being closed, and prove Lemma \ref{eigenvalue tail - 1}.

\subsubsection{Some useful facts}
\label{sec:Some useful facts}
Before embarking on the proof of Lemma \ref{iter}, we will show in this section that with high probability, assumption \eqref{eq:typical} holds and the choice of $z$ in Lemma \ref{iter} exists.
\begin{defn}
\label{def:Phi}
For any \rf{finite} $U \subset \Z^d$, let $\Phi_U$ be the $\ell^1$-normalized principal eigenfunction of $P|_{U}$, the transition matrix of the random walk restricted to $U$.	
\end{defn}
 The following lemma will be used repeatedly, for instance,  to bound $\Phi_U(v)$ at sites $v$ close to the boundary of $U$, or to find sites $v$ from where the walk cannot exit $U$ too quickly.
\begin{lemma}
For \rf{finite $U \subset \Z^d$ and} $t \in \N$,	
\begin{equation}
\label{Eigenfunction Value Bound}
	\sum_{v \in U} \Phi_U(v) \cdot \Pr^v( \tau_{U^c} \leq t) = 1 - \lambda_U^t\,.
\end{equation}	
\end{lemma}
\begin{proof}Let $\mathbf{1} = (1,1,\dots,1) \in \R^U$, then
\begin{equation*}
		\sum_{v \in U} \Phi_U(v) \cdot \Pr^v( \tau_{U^c} \leq t)  = 1 -  \langle\Phi_U, (P|_U)^t\mathbf{1} \rangle= 1- \lambda_U^t\,. \qedhere
	\end{equation*}	
\end{proof}

Recall the definition of $\emp_n(\epn)$ in Definition \ref{def:empty} and \eqref{eq:def-epn}. The following result says that neither $\tos$ nor $\emp_n(\epn)$ can be too large in a typical obstacle configuration, namely, \eqref{eq:typical} holds.

\begin{lemma}
\label{typical}
Let $\lm \in (C, \rad^{1/2})$ for some large constant $C$. There exists a constant $C_6>1$ depending only on $(d,p)$ such that \eqref{eq:typical} holds with $\P$-probability at least $1- n^{-10d}$.
\end{lemma}
\begin{proof}
Since Lemma \ref{empty} gives the second inequality in \eqref{eq:typical}, it suffices to show that the first inequality in \eqref{eq:typical} holds with high probability.
Consider boxes of the form $ K(v,\lm), v \in (2\lm+1)\Z^d$. We can partition these boxes into $10^d$ groups $\{K(v,\ell):v \in (2\ell + 1)(10\Z^d + i)\}$, for $i \in \{0,1,\dots,9\}^d$ so that the distance between any two boxes in the same group is at least $10\ell$. Recalling the definition of \trueop\ box in \eqref{eq:truly open}, we have that within each group the events that each box is \trueop\ are mutually independent, and have probability less than $e^{-c\ell^d}$ by Lemma \ref{toprob}. Note that on the event $\{|B(x,C_6 \rad) \setminus \tos |<\rad^{d}\}$, there exists a group where there are at most $\rad^d/[10^d\cdot(2\ell +1 )^d]$ many non-\trueop\ box that intersect $B(x,C_6 \rad)$. Also, the number of boxes that intersect $B(x,C_6 \rad)$ in each group is at least
$ |B(x,C_6 \rad)|/[10^d\cdot(2\ell +1 )^d]\,.$ It follows from large deviation estimates for sums of i.i.d. Bernoulli random variables (in each group) and a union bound over those groups that
$$ \P\big(|\tos \cap B(x,C_6 \rad)| \geq |B(x,C_6 \rad)|- \rad^d\big) \leq n^{-20d}\,,$$
for $C_6$ and $\ell$ sufficiently large.
Then a union bound over $ x \in B(0,(\log n)^{C_1})$ yields
$$ \P\Big(\min_{x \in B(0,(\log n)^{C_1})}|B(x,C_6 \rad) \setminus \tos| \geq\rad^{d}\Big) \geq 1- n^{-15d}\,.$$
This completes the proof of Lemma \ref{typical}.
\end{proof}

The following result says that under the assumptions in Lemma \ref{iter}, there exists $z$ such that $B(z, C_6\rad)$ contains enough \trueop\ boxes (for some $b_2$ depending on $b_1$ as determined by \eqref{eq:number of truly open site}).
\begin{lemma}\label{tl0332123}
Let $\lm \in (C, \rad^{1/2})$ for some large constant $C$. There exists a constant $c_3 = c_3(d) \in (0,1)$ such that for any $\consa \geq 1$, if $\lambda_\plr \geq 1 - \consa \rad^{-2}$ and $n$ is sufficiently large, then there exists $x \in  B(0,(\log n)^{C_1})$ such that 
\begin{equation}
\label{eq:number of truly open site}
	|\tos \cap B(x,C_6 \rad)| \geq c_3\consa^{-d/2}\rad^d\,.
\end{equation}
\end{lemma}
The following lemma is needed to prove Lemma \ref{tl0332123}.
\begin{lemma}
If $K(x,\lm)$ is not a \trueop\ box, then for any starting point $ u \in K(x,\lm)$, the survival probability up to $\ell^2$ steps is less than $1/2$, namely,
\begin{equation}
\label{eq:not truly open - killed}
	\Pr^u(\tau_{\plr^c} > \ell^2) \leq 1/2\,.
\end{equation}
\end{lemma}
\begin{proof}
We first note that for any $ u \in K(x,\lm)$,
\begin{equation*}
	\Pr^u(\tau_{\plr^c} > \ell^2) \leq \Pr^u(\max_{t \in [0,\ell^2]}|S_t - u|_\infty \geq 3\ell) +  \Pr^u(S_{[0,\lm^2]} \subset( K(x,4\ell)\setminus \ob))\,.
\end{equation*}
If $K(x,\lm)$ is not a \trueop\ box, then the definition of the \trueop\ boxes (Definition \ref{def:truly open}) implies
$$\Pr^u(S_{[0,\lm^2]} \subset( K(x,4\ell)\setminus \ob)) \leq 1/10\,.$$
In addition, the reflection principle yields
$$ \Pr^u(\max_{t \in [0,\ell^2]}|S_t - u|_\infty \geq 3\ell)  \leq d\cdot 2\Pr(|S_{\ell^2} \cdot \mathbf{e}_1| \geq 3 \ell) \leq 2d \cdot \frac{\ell^2/d}{9\ell^2}=2/9\,.$$
Combining the previous three inequalities gives \eqref{eq:not truly open - killed}.
\end{proof}

\begin{proof}[\bf Proof of Lemma \ref{tl0332123}]
Since by assumption $\lambda_\plr \geq 1 - b\rad^{-2}$,
\eqref{Eigenfunction Value Bound} implies that
\begin{equation*}
		\sum_{v \in \plr} \Phi_\plr(v) \cdot \Pr^v( \tau_{\plr^c} \leq 10^{-3}\rad^2/\consa) \leq 1- \lambda_\plr^{10^{-3}\rad^2/\consa} \leq  1 - e^{-\lceil 10^{-3}\rad^2/\consa \rceil \consa \rad^{-2}} \leq 1/100\,.
\end{equation*}	
Hence, there exists $u \in \plr$ such that
$$ \Pr^u(\tau_{\plr^c} \leq  \lceil 10^{-3}\rad^2/\consa \rceil) \leq  1/100\,.$$
On the other hand, by \eqref{eq:not truly open - killed} if the random walk hit $\tos^c$, then it will get killed with probability at least $1/2$ in next $\lm^2$ steps. Since $\ell^2 \leq \rad \leq 10^{-4}\rad^2/\consa$ for sufficiently large $n$, we get
$$ \Pr^u(\tau_{\plr^c} \leq \lceil 10^{-3}\rad^2/\consa \rceil)   \geq \Pr^u(S_{ \lceil 10^{-4}	\rad^2/\consa \rceil} \in \tos^c) /2\,.$$
Combining the previous two  inequalities gives
$$ \Pr^u(S_{ \lceil 10^{-4}	\rad^2/\consa \rceil} \in \tos^c)  \leq 1/50\,.$$
Since we assumed $\consa \geq 1$ and $C_6 \geq 1$ (chosen in Lemma \ref{typical}), we have
\begin{equation*}
	\Pr^u(S_{ \lceil 10^{-4}	\rad^2/\consa \rceil} \not \in B(u,C_6 \rad)) \leq 1/50\,.
\end{equation*}
Combining the previous two inequalities with the local limit theorem for the random walk $S$ gives
$$24/25\leq \Pr^u(S_{ \lceil 10^{-4}	\rad^2/\consa \rceil} \in \tos \cap B(u,C_6 \rad)) \leq |\tos \cap B(u,C_6 \rad)|  \cdot C(\rad b^{1/2})^{-d}\,.$$
This yields \eqref{eq:number of truly open site}.
\end{proof}

\subsubsection{Proof of Lemma \ref{iter} and its corollaries}
\label{sec:proof of iter}
\begin{proof}[\bf Proof of Lemma \ref{iter}]
We need to show that for each for each $z \in  B(0,(\log n)^{C_1})$ such that $|\tos \cap B(z,C_6 \rad)| \geq (\consa_2\rad)^d$, we can find a truly open box $K(x,\ell)$ such that filling $K(x, 10\ell)$ with obstacles will not decrease $\lambda_\plr$ too much, namely, \eqref{eq:itereq} holds. We will find such an $x$ near the boundary of $\tos$. The change in $\lambda_\plr$ can then be shown to be small because $\Phi_\plr$ is small near $x$. (Recall that $\Phi_\plr$ is the $\ell^1$-normalized principal eigenfunction of $P|_{\plr}$, the transition matrix of the random walk restricted to $\plr$.)

 Denote by $|\Phi_{\plr}|_2$ the $\ell^2$-norm of $\Phi_\plr$, namely $|\Phi_{\plr}|_2^2 = \sum_{x \in \Z^d}\Phi_\plr(x)^2$. By Lemma \ref{eigenvalue-difference} in the appendix, for any $x \in \plr$ with $\sum_{u \in K(x,11\lm)} \Phi_{\plr}^2(u) \leq |\Phi_{\plr}|_2^2/2$, we have
\begin{equation}
\label{eq:ds23-2}
 	\lambda_\plr - \lambda_{\plr \setminus K(x,10\lm)} \leq  4\sum_{u \in K(x,11\lm)} \Phi_{\plr}^2(u)/|\Phi_{\plr}|_2^2\,.
 \end{equation}

To bound the right hand side of \eqref{eq:ds23-2}, we first show that the assumption \eqref{eq:typical} implies
\begin{equation}
\label{eq:ds23-1}
	|\Phi_\plr|_2^2 \geq c\rad^{-d}\,.
\end{equation}
To this end, let $\Omega_\epn := \{v \in \Z^d: \Phi_\plr(v) \geq \epn \rad^{-d}\}$. It was shown in \cite[Lemma 5.5]{DX18} that $\sum_{v \not \in \emp_n(\epn)} \Phi_\plr(v) \leq C \epn$ for some constant $C$, which implies $|\Omega_\epn \setminus \emp_n(\epn)| \leq C\rad^d$. Since $|\emp_n(\epn)| \leq C_2 \rad^{d}$,  we get $|\Omega_\epn| \leq C \rad^d$ and $$\sum_{v \in \Omega_\epn} \Phi_\plr(v) \geq \sum_{v \in \Omega_\epn \cup \emp_n(\epn)} \Phi_\plr(v)  -|\emp_n(\epn)|\cdot \epn \rad^{-d} \geq 1- C \epn \geq 1/2\,.$$ Then \eqref{eq:ds23-1} follows from
	$|\Omega_\epn| \cdot \sum_{v \in \Omega_\epn} \Phi^2_\plr(v) \geq \Big(\sum_{v \in \Omega_\epn} \Phi_\plr(v) \Big)^2$.

Now, suppose $z \in  B(0,(\log n)^{C_1})$ satisfies $|\tos \cap B(z,C_6 \rad)| \geq (\consa_2\rad)^d$. By \eqref{eq:ds23-2} and \eqref{eq:ds23-1}, to prove \eqref{eq:itereq}, it suffices to find a \trueop\ $K(x,\lm)$ with $x \in B(z,20\rad)$ such that for some constant $C>0$,
\begin{equation}
\label{eq:150709}
	\sum_{u \in K(x,11\ell)} \Phi_\plr^2(u)  \leq C \consa_1^2\consa_2^{-2(d-1)}\rad^{-2(d+1)}\lm^{2(d+2)}\,.
\end{equation}
 Heuristically, $\Phi_\plr$ is large on $\tos$ and small on $\tos^c$. So we expect that such a \trueop\ box can be found near the boundary of $\tos$.

We define the outer boundary for $D \subseteq \Z^d$ by \begin{equation}
	\begin{split}
		\partial D &:= \{x\in D^c: |x-y|_1 = 1 \mbox{ for some } y\in D\}\,,
	\end{split}
\end{equation}
and denote by $A$ the points in $\Z^d$ which are close to $\tos^c$: $$A = \{u \in \Z^d : \exists \ v\in \tos^c \ s.t. \ |u-v|_\infty \leq 100 \lm\}\,.$$
For any starting point in $u \in A$, since $\tos^c$ is a union of boxes of side length $\ell$, the probability that the random walk hits $\tos^c$ within $\ell^2$ steps is uniformly bounded away from $0$. Recalling \eqref{eq:not truly open - killed}, which says that starting from any point in $\tos^c$, with probability at least $1/2$, the random walk will be killed in $\ell^2$ steps, we get for some constant $c' = c'(d)$, $$\Pr^u(\tau_{\plr^c} \leq 2\lm^2) \geq c'\,.$$
Then \eqref{Eigenfunction Value Bound} and the assumption $\lambda_\plr \geq 1 - \consa_1 \rad^{-2}$ implies
\begin{equation}
\label{eq:sumA}
\sum_{u \in A}\Phi_\plr(u)c'\leq  1- \lambda_\plr^{2\lm^2} \leq C\consa_1\lm^2\rad^{-2}\,.
\end{equation}
We claim that $A$ contains many truly open boxes. Indeed, the cardinality of $A\cap B(z, C_6\rad)$ can be bounded from below in terms of the cardinality of $\partial \tos \cap B(z, C_6\rad)$.  Since $|\tos \cap B(z,C_6 \rad)| \geq (\consa_2\rad)^d$, $|B(z,C_6 \rad) \setminus \tos| \geq  \rad^{d}$, and $\consa_2 \in (0,1)$, Lemma \ref{isolemma} implies that
$$|\partial \tos \cap B(z,C_6 \rad)| \geq c (\consa_2\rad)^{d-1}\,.$$
Then we can choose $c \lm^{-d}(\consa_2\rad)^{d-1}$ many \trueop\ boxes such that that the boxes of side length $22\lm +1$ centered at these boxes are disjoint and also in $A$. Combined with \eqref{eq:sumA}, it implies that there exists a \trueop\ box $K(x,\lm)$ with $x \in B(z,C_6 \rad)$ and
$$\sum_{u \in K(x,11\ell)} \Phi_\plr(u) \leq \frac{C\consa_1\lm^2\rad^{-2}}{c \lm^{-d}(\consa_2\rad)^{d-1}} \leq  C \consa_1\consa_2^{-(d-1)}\rad^{-(d+1)}\ell^{d+2}\,.$$ Hence
$$\sum_{u \in K(x,11\ell)} \Phi_\plr^2(u)  \leq \Big(\sum_{u \in K(x,11\ell)} \Phi_\plr(u) \Big)^2 \leq C \consa_1^2\consa_2^{-2(d-1)}\rad^{-2(d+1)}\ell^{2(d+2)}\,.$$
Thus \eqref{eq:150709} follows and this completes the proof of Lemma \ref{iter}.
\end{proof}

In the following lemma, we apply Lemma \ref{iter} repeatedly to remove a number of \trueop\ boxes, while $\lambda_\fregion$ only decreases slightly.

\begin{lemma}
\label{Remove Truly open set}

Assume \eqref{eq:typical} holds and $\lambda_\plr \geq 1 - \consa \rad^{-2}$ for some $\consa \geq 1$. There \rf{exist constants $C, C_7 >2$} such that the following holds\rf{: For all $\lm \in (C, \rad^{1/2})$}, there exist $z \in B(0,(\log n)^{C_1})$ and $\{x_m\}_{m=1}^{M_{\ell,b}} \subset B(z,C_6\rad)$ with
\begin{equation}
\label{eq:mMAx}
M_{\ell,b} := C_7^{-1}\consa^{-d-2}\lm^{-2d-2} \rad^d
\end{equation}
such that $\{K(x_m,\lm)\}_{m=1}^{M_{\ell,b}}$ are \trueop, $\{K(x_m,5\lm)\}_{m=1}^{M_{\ell,b}}$ are disjoint, and for any $1 \leq m \leq M_{\ell,b}$,
	\begin{equation}
	\label{eq:removet-lambda}
		\lambda_\plr - \lambda_{\plr \setminus \bigcup_{j=1}^{m}} K(x_j,10\lm) \leq \delta(m,\lm,b)\,,
	\end{equation}
where $\delta(m,\lm,b): = C_7m \consa^{3-d}\lm^{2(d+1)}\rad^{-d-2}$.
\end{lemma}
\begin{proof}
First note that since $\lm$ is sufficiently large and $\lm\leq \rad^{1/2}$, by Lemma \ref{tl0332123}, the assumption $\lambda_\plr \geq 1 - \consa \rad^{-2}$ implies that we can choose $z_\plr \in B(0,(\log n)^{C_1})$ such that
\begin{equation}
\label{eq:zplr}
	|\tos \cap B(z_\plr,C_6 \rad)| \geq c_3\consa^{-d/2}\rad^d\,.
\end{equation}
We will choose the $x_i$'s inductively by repeatedly applying Lemma \ref{iter} such that $x_1 \dots x_{M_{\ell,b}} \in B(z_\plr,C_6\rad)$, $K(x_1,\lm),\dots,K(x_{M_{\ell,b}},\lm)$ are \trueop, $K(x_1,5\lm),\dots,K(x_{M_{\ell,b}},5\lm)$ are disjoint and for $1 \leq i \leq {M_{\ell,b}}$,
\begin{equation}
\label{eq:iter-lambda}
	\lambda_{\plr \setminus \bigcup_{j=1}^{i} K(x_j,10\lm)} \geq \lambda_{\plr \setminus \bigcup_{j=1}^{i-1} K(x_j,10\lm)} - C_7 \consa^{3-d}\lm^{2(d+1)}\rad^{-d-2}\,.
\end{equation}
 Set
\begin{equation}
\label{eq:parameters}
	(\consa_1,\consa_2,z) = (2\consa,(c_3/2)^{1/d}\consa^{-1/2},z_\plr)\,,
\end{equation}
For $i=1$, we apply Lemma \ref{iter} with parameters in \eqref{eq:parameters}. We now verify the conditions in Lemma \ref{iter}. Firstly, \eqref{eq:typical} is satisfied by assumption. Secondly, we assumed $\lambda_\plr \geq 1 - \consa \rad^{-2}$ and thus $\lambda_\plr > 1 - \consa_1 \rad^{-2}$. Lastly, \eqref{eq:zplr} gives $|\tos \cap B(z,C_6\rad)| \geq c_3b^{-d/2}\rad^d > (b_2\rad)^d$. Hence Lemma \ref{iter} implies that there exists $x_1 \in B(z,C_6\rad)$ such that $K(x_1,\ell)$ is \trueop\ and \eqref{eq:iter-lambda} holds for $i=1$ and sufficiently large $C_7$.

Suppose that we have chosen $x_1,\dots,x_i$ for $1 \leq i \leq M_{\ell,b}-1$ with the aforementioned properties. We apply Lemma \ref{iter} with the same parameters as in \eqref{eq:parameters} to $\plr \setminus \bigcup_{j=1}^{i}K(x_j,10\lm)$ in place of $\plr$. We now verify the conditions in Lemma \ref{iter}.

Firstly, since $\tos$ and $\emp_n$ are non-increasing as we close $\bigcup_{j=1}^{i}K(x_j,10\lm)$ in $\plr$, \eqref{eq:typical} still holds.

Secondly, combining the hypothesis \eqref{eq:iter-lambda} and the assumption $\lambda_\plr \geq 1 - \consa \rad^{-2}$ yields
$$ 	\lambda_{\plr \setminus \bigcup_{j=1}^{i} K(x_j,10\lm)} \geq \lambda_{\plr } - i \cdot C_7 \consa^{3-d}\lm^{2(d+1)}\rad^{-d-2} \geq 1 - 2\consa\rad^{-2} = 1- \consa_1 \rad^{-2}\,,$$
where in the last inequality we used $i \leq M_{\ell,b}$ and $b \geq 1$.

Lastly, closing $\bigcup_{j=1}^{i}K(x_j,10\lm)$ will at most affect whether sites in $\bigcup_{j=1}^{i}K(x_j,20\lm)$ are \trueop\ or not. Hence, the reduction in the volume of \trueop\ box volume is at most $i(40\lm+1)^d$ and hence (with $\plr$ replaced by $\plr \setminus \bigcup_{j=1}^{i}K(x_j,10\lm)$) for sufficiently large $C_7$,
$$ |\tos \cap B(z,C_6\rad)| \geq c_3\consa^{-d/2}\rad^d - {M_{\ell,b}} (40\lm +1)^d \geq c_3\consa^{-d/2}\rad^d/2 = (\consa_2 \rad)^d\,.$$

Therefore Lemma \ref{iter} implies that there exists $x_{i+1} \in B(z,C_6\rad)$ such that $K(x_{i+1},\ell)$ is \trueop\ in $ \plr \setminus \bigcup_{j=1}^{i}K(x_j,10\lm)$ and \eqref{eq:iter-lambda} holds for $i+1$. Also, $K(x_{i+1},\ell)$ is \trueop\ in $ \plr  \setminus \bigcup_{j=1}^{i}K(x_j,10\lm)$ implies that $K(x_{i+1},\ell)$ is \trueop\ in $ \plr$, and it is disjoint from $ \bigcup_{j=1}^{i}K(x_j,10\lm)$. Hence $K(x_1,5\lm),\dots,K(x_{i+1},5\lm)$ are disjoint.
This completes the  proof of Lemma \ref{Remove Truly open set}.
\end{proof}

Next we estimate the probability gain achieved by closing the \trueop\ boxes $\{K(x_i,\lm)\}_{i=1}^{M_{\ell,b}}$ identified in Lemma \ref{Remove Truly open set}.

\begin{lemma}
\label{mllemma}
Let $M_{\ell,b}$ and $\delta(m,\lm,b)$ be as in Lemma \ref{Remove Truly open set}.
There exists a constant $\rf{C_8>0}$ such that the following holds: Suppose $\beta \geq 1 - \consa \rad^{-2}$ for some constant $\consa \geq 1$, \rf{$\ell < \rad^{1/2}$}, $m \leq M_{\ell,b}$ and
\begin{equation}
\label{eq:mld-con}
 \lm^d \geq C_8\log (\rad^d/m)\,.
\end{equation}
Then
		\begin{equation}
	\label{eq:lem3.2-lm}
		\P(\lambda_\plr \geq \beta) \leq e^{-cm \lm^d}\P(\lambda_\plr \geq \beta -\delta(m,\lm,b)) + n^{-10d}\,.
	\end{equation}
\end{lemma}
\begin{proof}
We will consider an operation that changes some \trueop\ boxes to typical obstacle configurations, which maps the event $\{\lambda_\plr \geq \beta\}$ to $\{\lambda_\plr \geq \beta -\delta(m,\lm,b)\}$, allowing us to bound the probability ratio of these two events.

To this end, we define $\tos_{U^c}$ and $\emp_{U^c}$ for general $U\subset B(0,2(\log n)^{C_1})$ by regarding $U^c=\ob$, and for any $\beta> 1 - \consa \rad^{-2}$, consider two %sets of obstacle configurations in
classes of subsets of $B(0,2(\log n)^{C_1})$:
\begin{align*}
\mathscr U(\beta) &= \{U \subset B(0,2(\log n)^{C_1}): \lambda_{U \cap B(0,(\log n)^{C_1})} \geq \beta\}\,,\\
	\mathscr G &= \{U \subset B(0,2(\log n)^{C_1}):\min_{x \in B(0,(\log n)^{C_1})}|B(x,C_6 \rad) \setminus \tos_{U^c}| \geq \rad^{d},|\emp_{U^c}(\epn)| \leq C_2 \rad^d\}\,.
\end{align*}
Then, denoting $\plr_+:=B(0, 2(\log n)^{C_1})\setminus \ob$, we can rewrite the event $\{\lambda_\plr \geq \beta\}\cap \{\eqref{eq:typical}\text{ holds}\}$ as $\{\plr_+\in \mathscr U(\beta) \cap \mathscr G\}$.

\rf{Since our condition implies $\ell \in (C_8^{1/d}, \rad^{1/2})$} and $m \leq M_{\ell,b}$, we \rf{can} define a map $\Xi_m$ for all $U \in \mathscr U(\beta) \cap \mathscr G$ by
\begin{equation}
\label{eq:def-Xi}
	\Xi_m ( U ):=\bigcup_{j=1}^{m} K(x_j,5\lm)\,,
\end{equation}
where $x_1,\dots,x_m$ are chosen as in Lemma \ref{Remove Truly open set} depending on $U$ (make arbitrary choice when $x_i$'s are not unique.)
The idea of the proof is the following. For each $U \in \mathscr U(\beta) \cap \mathscr G$, we change the obstacle configuration in $\Xi_m(U)$ to typical configurations. The image of $\mathscr U(\beta) \cap \mathscr G$ has much higher probability under the law of $\plr_+$ than that of $\mathscr U(\beta) \cap \mathscr G$ itself, because $\Xi_m(U)$ contains $m$ \trueop\ boxes at a large probability cost. Combined with \eqref{eq:removet-lambda}, which yields that the image of $\mathscr U(\beta) \cap \mathscr G$ is a subset of $\mathscr U(\beta - \delta(m,\lm,b))$, this gives the desired result.

We now rigorously implement this idea. We first define an equivalence relation on $\mathscr U(\beta) \cap \mathscr G$ by
$$ U \sim U' \iff \Xi_m(U) =  \Xi_m(U'),U \setminus \Xi_m(U) = U' \setminus\Xi_m(U')\,,  $$
and denote the equivalence class for $U$ in $\mathscr U(\beta) \cap \mathscr G /{\sim}$ by $[U]$, namely
$$ [U] := \{U' \subset B(0,2(\log n)^{C_1}):  U' \sim U\}\,.$$
Consider the map
$$\varphi ([U]) = \{V: V \setminus \Xi_m(U) =  U \setminus \Xi_m(U)\}\,,$$
which contains modifications of $U$ by allowing arbitrary configurations on $\Xi_m(U)$ as long as the set $\Xi_m(U)$ does not change.
Applying Claim \ref{112341} below to two families of events $(\{\plr_+ \in [U]\})_{[U] \in \mathscr U(\beta) \cap \mathscr G /{\sim}}$ and $\{\plr_+ \in \varphi([U])\}_{[U] \in \mathscr U(\beta) \cap \mathscr G /{\sim}}$, we obtain that

\begin{equation}
\label{eq:4.21}
	\frac{\P\big(\plr_+ \in \bigcup_{[U]}\varphi([U]) \big)}{\P(\plr_+ \in \mathscr U(\beta) \cap \mathscr G)} \geq \inf_{[U]} \frac{\P(\plr_+ \in \varphi([U]))}{\P(\plr_+ \in [U])}  \Big/\sup_{V \subset B(0,2(\log n)^{C_1})} \sum_{[U]} \11_{V \in \varphi([U])} \,.
\end{equation}
\begin{claim}
\label{112341}
	Let $(E_i)_{1\leq i\leq k}$ and $(F_i)_{1\leq i\leq k}$ be two families of events. Then we have

$$\frac{\P(\bigcup_i F_i)}{\P(\bigcup_i E_i)} \geq \frac{\inf_{i} \P(F_i)/\P(E_i)}{\sup_\omega \sum_i \11_{\omega \in F_i}}\,. $$
\end{claim}
\begin{proof}This follows from
	$$ \inf_{i} \frac{\P(F_i)}{\P(E_i)} \cdot \sum_i \P(E_i) \leq \sum_i \P(F_i) = \E \Big[ \sum_i \11_{\omega \in F_i}\Big] \leq \sup_\omega \sum_i \11_{\omega \in F_i}\P\Big(\bigcup_i F_i\Big)\,. \qedhere$$
\end{proof}
Since Lemma \ref{Remove Truly open set} gives
\begin{equation*}
\label{eq:5177}
	\bigcup_{[U]}\varphi([U]) \subset \mathscr U(\beta - \delta(m,\lm,b))
\end{equation*}
and Lemma \ref{typical} yields $\P(\plr_+\in\mathscr G) \geq 1- n^{-10d}$, we can bound the left hand side of \eqref{eq:4.21} by
$$ \frac{\P\big(\plr_+ \in \bigcup_{[U]}\varphi([U]) \big)}{\P(\plr_+ \in \mathscr U(\beta) \cap \mathscr G)} \leq \frac{\P(\lambda_\plr \geq \beta - \delta(m,\ell,b))}{\P(\lambda_\plr \geq \beta)- n^{-10d}}\,.$$
Therefore, to prove \eqref{eq:lem3.2-lm}, it suffices to show that for some constant $c = c(d,p)$,
\begin{equation}
	\inf_{[U]} \frac{\P(\plr_+ \in \varphi([U]))}{\P(\plr_+ \in [U])}  \Big/\sup_{V \subset B(0,2(\log n)^{C_1})} \sum_{[U]} \11_{V \in \varphi([U])} \geq e^{cm\ell^d}\,.
\end{equation}

We first prove that there exists a constant $c' = c'(d,p)$ such that
\begin{equation}
\label{eq:231241-1}
	\inf_{[U]} \frac{\P(\plr_+ \in \varphi([U]))}{\P(\plr_+ \in [U])} \geq e^{c'm\ell^d}\,.
\end{equation}
For any fixed $U \in (\mathscr U(\beta) \cap \mathscr G)$, $\Xi_m(U) $ is a union of $m$ boxes (defined in \eqref{eq:def-Xi}), which we denote by $ K(x_i,5\lm)$ for $1\leq i \leq m$. Then
\begin{equation*}
\begin{split}
\P(\plr_+ \in [U]) &\leq \P(\plr_+ \setminus \Xi_m(U) = U \setminus \Xi_m(U),K(x_i,\lm) \text{ for all $i \leq m$ are \trueop})\,.
\end{split}
\end{equation*}
Note that the events $\{\plr_+ \setminus \Xi_m(U) = U \setminus \Xi_m(U)\}$, $\{K(x_i,\lm) \text{ is \trueop})\}$ for $i \leq m$ are mutually independent as they depend on obstacle configurations on disjoint regions. Then by Lemma \ref{toprob},
\begin{equation*}
\begin{split}
\P(\plr_+ \in [U]) &\leq \P(\plr_+ \setminus \Xi_m(U) = U \setminus \Xi_m(U)) \cdot \P(K(0,\lm) \text{ is \trueop})^m \\
&\leq e^{-cm\lm^d} \P(\plr_+ \setminus \Xi_m(U) = U \setminus \Xi_m(U))\\
& = e^{-cm\lm^d} \P\big(\plr_+ \in \varphi([U]) \big)\,.
\end{split}
\end{equation*}
This gives \eqref{eq:231241-1}.

Now, it only remains to prove that for the constant $c'>0$ as in \eqref{eq:231241-1},
\begin{equation}
\label{eq:231241-2}
	\sup_{V \subset B(0,2(\log n)^{C_1})} \sum_{[U]} \11_{V \in \varphi([U])} \leq e^{c'm\ell^d/2}\,.
\end{equation}
To this end, note that
$$ \sum_{[U]} \11_{V \in \varphi([U])} = |\{[U]: V \setminus \Xi_m(U) =  U \setminus \Xi_m(U)\}|\,,$$
where the cardinality of the set of such $[U]$ is bounded by the number of possible choices of $\Xi_m(U) = \bigcup_{i=1}^m K(x_i, 5\lm)$ with $x_1,\dots,x_m$ in $B(z,C_6 \rad)$ for some $z \in B(0,2(\log n)^{C_1})$. Therefore, we have
$$ \sum_{[U]} \11_{V \in \varphi([U])}  \leq |B(0,2(\log n)^{C_1}) | \cdot \binom{|B(0,C_6 \rad)|}{m} \leq 2^d(\log n)^{C_1d} (e(2C_6 \rad)^d/m)^m\,. $$
Since $m \leq M_{\ell,b}$ (defined as in \eqref{eq:mMAx}) ensures $\rad^d/m \geq 2$, we can further bound the right hand side above by
$$\exp(C \log \log n+ Cm\log(\rad^d/m))$$ for some large constant $C>0$. Also, by \eqref{eq:def-r} and $\rad^d/m \geq 2$, we have
$\log \log n \leq Cm\log(\rad^d/m) $
for some large constant $C>0$. Therefore, the assumption \eqref{eq:mld-con} implies there exists a constant $C'$ such that
$$ C \log \log n+ Cm\log(\rad^d/m) \leq C_8^{-1}C'm\ell^d\,.$$
For $C_8$ sufficiently large, this implies \eqref{eq:231241-2}, which completes the proof of Lemma \ref{mllemma}.
\end{proof}
\subsubsection{Proof of Lemma \ref{eigenvalue tail - 1} }
\label{sec:tail-1}
We only need to apply Lemma \ref{mllemma} with appropriate choices of $\lm$ and $m$. Recall from \eqref{eq:e-range-1} that $\epsilon \in ((\log\log n)^4\rad^{-d},c_\consa)$  is an arbitrary number for some small constant $c_\consa$ to be determined. We let
\begin{equation}
\label{eq:mandl}
	\begin{split}
	  \lm &= \lfloor \Theta[\log (1/\epsilon)]^{1/d}\rfloor\,,\\
		m &
		= \lfloor \theta \cdot \epsilon  (\log (1/\epsilon))^{-3} \cdot \rad^d  \rfloor\,,
	\end{split}
\end{equation}
where $\Theta$ and $\theta$ are constants depending only on $(d,p,b)$ to be determined.

First, we verify that all the conditions in Lemma \ref{mllemma} hold. Since $\epsilon \geq(\log\log n)^4\rad^{-d}$, \eqref{eq:mandl} implies $\ell <\rad^{1/2}$ and $m \geq 1$. On the other hand, for all $\epsilon < c_b$ with $c_b = c_b(d,p,b,\theta,\Theta)$ sufficiently small, we have$$\epsilon^2\rad^d \leq m  \leq \epsilon \rad^d \cdot 2\Theta^{3d}\theta\cdot\ell^{-3d}\,.$$
Hence $m \leq M_{\ell,b}$ (defined in Lemma \ref{Remove Truly open set}), and $\lm^d \geq \Theta^d \log (\rad^d/m)/2$.
Then \eqref{eq:mld-con} follows by choosing $\Theta$ to be sufficiently large.
We thus know that all the conditions in Lemma \ref{mllemma} hold and hence this lemma yields
	\begin{equation*}
		\P(\lambda_\plr \geq \beta) \leq e^{-cm \lm^d}\P(\lambda_\plr \geq \beta -\delta(m,\lm,b)) + n^{-10d}\,.
	\end{equation*}
Recall $\delta(m,\lm,b)$ from Lemma \ref{Remove Truly open set}. Then \eqref{eq:mandl} yields that for $\theta$ and $c_b$ sufficiently small,
\begin{align*}
	\delta(m,\lm,b)\leq \epsilon \rad^{-2} \quad \text{and} \quad m\ell^d &\geq \epsilon  (\log (1/\epsilon))^{-3} \cdot \rad^d\,.
\end{align*}
We thus complete the proof of Lemma \ref{eigenvalue tail - 1}. \qed

\subsection{Proof of Lemma \ref{eigenvalue tail - 2 - not open}}
\label{sec:tail2}
As we have discussed in Section \ref{sec:Proof Outline}, we want to understand how much the eigenvalue will increase when we remove obstacles inside the ball $B(\ct{\plr},\rad)$. The difficulty is that, for $k \geq 0$, removing obstacles in $B_{\delta,k}$ (defined in \eqref{eq:def-b-dk}) may hardly increase $\lambda_\plr$, especially when there are many obstacles outside $B_{\delta,k}$  near its boundary and most obstacles in $B_{\delta,k}$ are also near the boundary. However, if we first remove all obstacles in the annulus $B_{\delta,k-1} \setminus B_{\delta,k}$ and then remove all obstacles in $B_{\delta,k}$ , then in the second step, the increase in the eigenvalue $\lambda_\plr$ can be bounded from below in terms of $|B_{\delta,k} \cap \ob|$ since all removed obstacles are in $B_{\delta,k-1}$ with distance at least $\delta 2^{-k}$ to the boundary of $B(\ct{\plr},\rad)$. The $\lj$ defined in \eqref{eq:def-J} ensures that a significant proportion of obstacles are in the bulk of the ball $B_{\delta, \lj-1}$:

\begin{lemma}
\label{obstaclelj}
	Let $\lj,c_5$ be as in \eqref{eq:def-J} and let $\lambda_*$ be as in \eqref{lambda*}. For any $\delta>0$, we assume $\lambda_\plr \geq \lambda_*$, \eqref{eq:empty prob E} holds, and $B(\ct{\plr},(1 - \delta)\rad) \cap \ob \not = \varnothing$. Then there exists a constant $C>0$ such that
	 \begin{equation}
  	\frac{|\ob \cap B_{\delta,\lj-1}|}{\rad^d}\leq C\epn^{1/4}c_5^{\lj-1}, \quad \frac{|\ob \cap B_{\delta,\lj}|}{|\ob \cap B_{\delta,\lj-1}|} \geq c_5\,.
  \end{equation}
\end{lemma}
\begin{proof}
	This result follows directly from the definition \eqref{eq:def-J} and \eqref{eq:obrad}.
\end{proof}

The following lemma gives a lower bound on how much the eigenvalue will increase if we remove all obstacles in $B_{\delta, \lj-1}$.
\begin{lemma}
\label{RemoveOb}
Suppose $\lambda_\plr \geq \btn$ and \eqref{eq:empty prob E} holds. (Recall the definition of $\lj$ in \eqref{eq:def-J} depending on $c_5$.) \rf{Then} there exist constants $c_5,\kappa \in (0,1)$ depending only on $(d,p)$ such that for $\delta = \rad^{-\kappa}$,
\begin{equation}
\label{eq:RemoveOB}
	\lambda_{\plr \cup B_{\delta,\lj-1}} - \lambda_{\plr \cup (B_{\delta,\lj-1} \setminus B_{\delta,\lj})}  \geq  \Big(\frac{|\ob \cap B_{\delta,\lj-1}|}{\rad^d}\Big)^{1 - 1/d} \rad^{-2}\,.
\end{equation}
\end{lemma}
Lemma \ref{RemoveOb} is proved by applying  Lemma \ref{RemoveB} in the appendix, which requires the following estimates.
\begin{lemma}
\label{eigenfunction value in the center}
Let $\lambda_*$ and $\Phi$ be defined as in Lemma \ref{lambda*} and Definition \ref{def:Phi}, respectively. Suppose $\lambda_\plr \geq \btn$ and \eqref{eq:empty prob E} holds. Then there exists a constant $b_1=b_1(d,p) \in (0,1)$ such that for all $U$ with $\plr \subset U \subset \plr \cup B(\ct{\plr},\rad)$, we have
	\begin{equation}
	\label{eq:eigenfunction value in the center}
		\sum_{u \in B(\ct{\plr},b_1\rad)} \Phi_U(u) \geq 1/2\,.
	\end{equation}
\end{lemma}
\begin{proof}
Recall the definition of $\emp_n(\epn)$ from Definition \ref{def:empty} and \eqref{eq:def-epn}. Note that
\begin{equation}
	B(\ct{\plr},b_1\rad)^c \subset A_1 \cup A_2 \cup A_3
\end{equation}
where
\begin{align*}
	&A_1 := B(\ct{\plr},(1 + \sqrt{d} \epn)\rad) \setminus B(\ct{\plr},b_1\rad)\,,\\
	&A_2 := \emp_n(\epn) \setminus B(\ct{\plr},(1 + \sqrt{d} \epn)\rad)\,,\\
	&A_3 := \big(\emp_n(\epn) \cup B(\ct{\plr},(1 + \sqrt{d}\epn)\rad) \big)^c\,.
\end{align*}
Hence
\begin{equation}
\label{eq:Uhalf-0}
	\sum_{u \in B(\ct{\plr},b_1\rad)} \Phi_U(u) \geq  1- \sum_{u \in A_3}\Phi_U(u) - |\Phi_U|_\infty (|A_1|+|A_2|)\,.
\end{equation}
We first prove that
\begin{equation}
\label{eq:Uhalf-1}
	\sum_{u \in A_3} \Phi_U(u) \leq C \epn \,.
\end{equation}
To this end, we notice that since $U \setminus B(\ct{\plr},\rad) = \plr \setminus B(\ct{\plr},\rad)$, it follows from the definition of $\emp_n(\epn)$ (which depends on $\plr$) that that $\emp_n(\epn) \setminus B(\ct{\plr},(1 + \sqrt{d} \epn)\rad)$  does not change if we change $\plr$ to $U$. Therefore, for every $u \in A_3$, there exists $u'$ such that $u \in K(u',\lfloor \epn \rad\rfloor)$ and $|U^c \cap K(u',\lfloor \epn \rad\rfloor)| \geq \epn |K(u',\lfloor \epn \rad\rfloor)|$. Hence by the local limit theorem
$$\Pr^u(\tau_{U^c} \leq (\epn \rad)^2)\geq c\epn\,.$$ Then \eqref{Eigenfunction Value Bound} gives
$$ \sum_{u \in A_3} \Phi_U(u) \leq C \epn^{-1}(1 - \lambda_U^{(\epn \rad)^2})\,.$$
Now, since $\plr \subset U$, we have $\lambda_U \geq \lambda_\plr \geq \btn$ and thus \eqref{eq:Uhalf-1} follows.

Next, by Lemma \ref{Philinfbd}, $\lambda_U \geq \btn$ implies $|\Phi_U|_\infty \leq C \rad^{-d}$.
Combined with \eqref{eq:Uhalf-0} and \eqref{eq:Uhalf-1}, this implies
\begin{equation}
\label{eq:934134-1}
	\sum_{u \in B(\ct{\plr},b_1\rad)} \Phi_U(u) \geq  \frac{2}{3} - C \rad^{-d}(|A_1|+|A_2|)\,.
\end{equation}
By \eqref{eq: Ball and E}, we have
\begin{equation}
\label{eq:934134-2}
	|A_1| + |A_2| \leq C(1 - b_1 + \sqrt{d}\epn + \epn^{1/4})\rad^d\,.
\end{equation}
Combining \eqref{eq:934134-1} and \eqref{eq:934134-2}, we see that \eqref{eq:eigenfunction value in the center} follows by letting $b_1$ be a constant sufficiently close to $1$.
\end{proof}

\begin{proof}[\bf Proof of Lemma \ref{RemoveOb}]
We apply Lemma \ref{RemoveB} with
$$ \tp = \plr \cup B_{\delta,\lj-1}, \qquad \tcp = \plr \cup (B_{\delta,\lj-1} \setminus B_{\delta,\lj})\,,$$
 $B_{R_1} = B_{\delta,\lj-1}, B_{R_2} =  B_{\delta,\lj}$ (defined in \eqref{eq:def-J}), and $B_{R_3} =B(\ct{\plr},b_1\rad)$ where $b_1$ is chosen as in Lemma \ref{eigenfunction value in the center}. It follows from $\lambda_\plr \geq \btn$ and Lemma \ref{eigenfunction value in the center} that the conditions \eqref{eq:eigenball} and \eqref{eq:eigenbd} in Lemma \ref{RemoveB} holds. Hence, by \eqref{eq:removeB} and the definition of $B_{\delta, \lj-1}$ in \eqref{eq:def-J}, we have
\begin{equation}
\label{eq:01815-1}
\begin{split}
		\lambda_{\plr \cup B_{\delta,\lj-1}} - \lambda_{\plr \cup (B_{\delta,\lj-1} \setminus B_{\delta,\lj})} &\geq \frac{c}{R_1^d(\log R_1)^{\11_{d=2}}}\Big(1 - \frac{R_2}{R_1} \Big)^{C_d} |\tp \setminus \tcp|^{(d-2)/d}\\
		& \geq \frac{c}{\rad^{d}(\log \rad)^{\11_{d=2}}} (4^{\lj} \delta)^{-C_d} |\ob \cap B_{\delta,\lj}|^{(d-2)/d}\,.
\end{split}
\end{equation}
By Lemma \ref{obstaclelj},
\begin{equation}
\label{eq:01815-2}
\begin{split}
		\rad^{-d}|\ob \cap B_{\delta,\lj}|^{(d-2)/d} &\geq c_5^{(d-2)/d}\rad^{-d}|\ob \cap B_{\delta,\lj-1}|^{(d-2)/d}\\
	& = c_5^{(d-2)/d}\rad^{-2}(|\ob \cap B_{\delta,\lj-1}|/\rad^d)^{(d-1)/d - 1/d}\\
	& \geq cc_5\epn^{-1/4d} \cdot c_5^{-\lj/d} \cdot \rad^{-2}(|\ob \cap B_{\delta,\lj-1}|/\rad^d)^{(d-1)/d}\,,
\end{split}
\end{equation}
where we have used \eqref{eq:obrad}.
Recall that we have set $\delta = \rad^{-\kappa}$ and that $\epn = \rad^{-c_2}$ as defined in \eqref{eq:def-epn}. Combining \eqref{eq:01815-1} and \eqref{eq:01815-2}, we complete the proof of \eqref{eq:RemoveOB} by choosing $c_5$ and $\kappa$ sufficiently small.
\end{proof}

\begin{proof}[\bf Proof of Lemma \ref{eigenvalue tail - 2 - not open}]
In light of Lemma \ref{RemoveOb}, we will bound the probability ratio in \eqref{eq:eigenvalue tail - 2 - not open} by considering the operation of removing all obstacles in $B_{\delta, \lj-1}$.
First note that the condition \eqref{eq:empty prob E}, $B_{\delta,k}$ and $\lj$ (defined in \eqref{eq:def-b-dk} and \eqref{eq:def-J}, respectively) only depend on $\ob \cap B(0,2(\log n)^{C_1})$. Therefore, for $\beta \geq \btn$ and $m \geq 1$, we define $$\mathscr U_{\beta,m} := \Big\{U \subset B(0,2(\log n)^{C_1}):\lambda_{U \cap B(0,(\log n)^{C_1})} \geq \beta,\eqref{eq:empty prob E} \text{ holds},	|B_{\delta,\lj-1} \setminus U| = m\Big\} $$
where \eqref{eq:empty prob E}, $B_{\delta,k}$, $\lj$ should be understood as if $\ob = U^c$.

Now we consider the map $\phi$ for $U \in \mathscr U_{\beta,m}$ that removes all obstacles in $B_{\delta,\lj-1}$, namely, $$\phi( U):= U \cup B_{\delta,\lj-1}\,.$$
Then by Lemma \ref{RemoveOb},
\begin{equation}
\label{eq:remove2}
	\bigcup_{U \in \mathscr U_{\beta,m}}\phi(U) \subset \Big\{U \subset B(0,2(\log n)^{C_1}): \lambda_{U \cap B(0,(\log n)^{C_1})} \geq \beta + \Big(\frac{m}{\rad^d}\Big)^{1 - 1/d} \rad^{-2}\Big\}\,.
\end{equation}
Recall that $\plr^+ := B(0,2(\log n)^{C_1}) \setminus \ob$. For every $U \in \mathscr U_{\beta,m}$, there are exactly $m$ closed sites in $B_{\delta,\lj-1}$, thus
$$\P(\plr^+ = U) = \Big(\frac{1-p}{p}\Big)^m \cdot \P(\plr^+ = \phi(U))\,.$$
Then by Claim \ref{112341}, we have that
\begin{align*}
	\P(\plr^+ \in \mathscr U_{\beta,m})
	&\leq \Big(\frac{1-p}{p}\Big)^m \cdot \max_{U \in \mathscr U_{\beta,m}}|\phi^{-1}(U)| \cdot \P(\plr^+ \in \bigcup_{U \in \mathscr U_{\beta,m}}\phi(U) )\\
	&\leq \Big(\frac{1-p}{p}\Big)^m \cdot \max_{U \in \mathscr U_{\beta,m}}|\phi^{-1}(U)| \cdot \P(\lambda_\plr \geq \beta + \big(\frac{m}{\rad^d}\big)^{1 - 1/d} \rad^{-2})\,,
\end{align*}
where in the last step we used \eqref{eq:remove2}. The multiplicity $\max_{U \in \mathscr U_{\beta,m}}|\phi^{-1}(U)|$ is bounded above uniformly over $U$ by the number of sets of $m$ points contained in a ball of radius $\rad$ centered at some point in $B(0,(\log n)^{C_1})$, namely,
$$\max_{U \in \mathscr U_{\beta,m}}|\phi^{-1}(U)| \leq |B(0,2(\log n)^{C_1})| \cdot \binom{|B(0,\rad)|}{m} \leq C \rad^{dC_1}\big(\frac{e(2\rad)^d}{m}\big)^m\,.$$
We complete the proof of \eqref{eq:eigenvalue tail - 2 - not open} by combining the preceding two inequalities.
\end{proof}

\subsection{Proof of Proposition \ref{po-clear}}
\label{sec:4.1}
We first note that Proposition \ref{po-clear} follows from the following result.\begin{claim}
\label{open or not open}
Let $\kappa>0$ be defined as in Lemma \ref{eigenvalue tail - 2 - not open}, and $\lambda_*$ as in \eqref{eq:lambda*lb}. Then for all $\beta \geq \btn$ and $n$ sufficiently large, 
			\begin{equation}
			\label{eq:open or not open}
			 \P(\lambda_\plr \geq \beta, B(\ct{\plr},\rad - \rad^{1-\kappa})\cap \ob \not = \varnothing)
			 \leq e^{-\rad^{1/3}} \cdot \P(\lambda_\plr \geq \beta )  + 4e^{-\rad}n^{-d }\,.
	\end{equation}
\end{claim}
Indeed, applying Claim \ref{open or not open} with $\beta = \lambda_*$ yields
\begin{equation}
\label{eq:0905}
	 \P(\lambda_\plr \geq \lambda_*, B(\ct{\plr},\rad - \rad^{1-\kappa})\cap \ob \not = \varnothing) \leq e^{-\rad^{1/3}} \P(\lambda_\plr \geq \lambda_*)+ 4e^{-\rad}n^{-d }\,.
\end{equation}
Combined with \eqref{eq:lambda*d}, this yields
$$ \P(\lambda_\plr \geq \lambda_*, B(\ct{\plr},\rad - \rad^{1-\kappa})\cap \ob \not = \varnothing) \leq (e^{-\rad^{1/3}} + 4e^{-\rad} (\log n)^{-c} )\P(\lambda_\plr \geq \lambda_*)\,,$$
which implies \eqref{eq:po-clear}.

Next, we prove Claim \ref{open or not open}.
We assume $\lambda_\plr \geq \beta$, and by Lemma \ref{empty}, we may also assume that \eqref{eq:empty prob E} holds. Then by Lemma \ref{Onecity-Ball}, this implies \eqref{eq:obrad}. Let $\kappa>0$ be defined as in Lemma \ref{eigenvalue tail - 2 - not open} and let $\delta = \rad^{-\kappa}$. Recall $B_{\delta,\lj-1}$ as in \eqref{eq:def-b-dk}, \eqref{eq:def-J}. Then \eqref{eq:obrad} gives $$|\ob \cap B_{\delta,\lj-1}| \leq |\ob \cap B(\ct{\plr},\rad)| \leq C\epn^{1/4} \rad^d\,.$$
Now, for each $m=|\ob \cap B_{\delta,\lj-1}|  \in [1, C\epn^{1/4} \rad^d]$, we denote $q = m/\rad^d$. Then Lemma \ref{eigenvalue tail - 2 - not open} yields
\begin{equation}
\label{eq:4.39}
		\P(\lambda_\plr \geq \beta , |\ob \cap B_{\delta,\lj-1}| = m,\eqref{eq:empty prob E}) \leq C \rad^{dC_1} (C /q)^{q \rad^d}\P(\lambda_\plr \geq \beta +  q^{1 - 1/d} \rad^{-2})\,,
\end{equation}
while applying Lemma \ref{eigenvalue tail - 1} with $\emm = q^{1 - 1/d}$ gives
\begin{align*}
	\P(\lambda_\plr \geq \beta +  q^{1 - 1/d} \rad^{-2}) &\leq \exp\big\{-(1 - 1/d)^{-3}q^{1-1/d}(\log(1/q))^{-3} \rad^d \big\}\P(\lambda_\plr \geq \beta) + n^{-10d}\,.
\end{align*}
Since $(1 - 1/d)^{-3} \geq 1$ and for sufficiently large $n$ we have
\begin{align*}
	C &\rad^{dC_1} (C /q)^{q \rad^d}	\cdot \exp\{-q^{1-1/d}(\log(1/q))^{-3} \rad^d\}\\
	  &\leq \exp\{C \log \rad + Cm \log (\tfrac{\rad^d}{m}) - m (\tfrac{\rad^d}{m})^{1/d}\log^{-3} (\tfrac{\rad^d}{m})\} \\
	&\leq \exp(-m^{1- 1/d} \rad^{1/2})\,,
\end{align*}
and $C \rad^{dC_1} (C /q)^{q \rad^d}\leq \exp(C\log \rad + C \rad^d q\log (\tfrac{1}{q})) \leq n$, we obtain from \eqref{eq:4.39} that
\begin{align*}
	\P(\lambda_\plr \geq \beta , |\ob \cap B_{\delta,\lj-1}| = m,\eqref{eq:empty prob E}) \leq e^{-\rad^{1/2}}\P(\lambda_\plr \geq \beta) + n^{1-10d}\,.
\end{align*}
Summing it over $1 \leq m \leq C\epn^{1/4} \rad^d$ yields
$$ \P(\lambda_\plr \geq \beta , \eqref{eq:empty prob E}) \leq e^{-\rad^{1/2}/2}\P(\lambda_\plr \geq \beta) + n^{-8d}\,.$$
We complete the proof of \eqref{eq:open or not open} by noticing that Lemma \ref{empty} yields that \eqref{eq:empty prob E} holds with probability at least $1- 3 n^{-d} e^{-\rad}$.
\section{Random Walk Localization}
\label{sec:localization}
In this section, we first collect a few survival probability estimates from \cite{DX18} in Section \ref{sec:sp-estimate}. Then we prove Lemmas \ref{pt-loc}, \ref{hit ball}, and \ref{loclemma2}, Corollary \ref{loclemma} in Section \ref{sec:transition probability}, and prove Lemma \ref{distribution} in Section \ref{sec:distribution}.
\subsection{Survival probability estimates}
\label{sec:sp-estimate}
The following lemma gives upper and lower bounds on the probability that the random walk stays in $\fregion$ for $t$ steps, which can be found in \cite[Lemma 6.3, Lemma 6.9]{DX18}.
\begin{lemma}
\label{stayinU}
Let $\hbn$ and $\fregion$ be as in Theorem \ref{thm:Ball confinement}, and let $\Phi$ be as in Definition \ref{def:Phi}. There exist constants $C,c>0$ such that the following holds with $\widehat\P$-probability tending to one as $n \to \infty$: For any $u \in \fregion$, $t \geq 0$,
	\begin{align}
		\label{eq:t-in-U-lb}&\Pr^u(\tau_{\fregion^c} > t) \geq c \Phi_\fregion(u)\rad^d \lambda_{\fregion}^t\,,\\
		\label{eq:t-in-U-ub}&\Pr^u(\tau_{\fregion^c} > t, S_t \in \hbn) \leq C \lambda_{\fregion}^t\,.
	\end{align}
\end{lemma}

The next lemma gives upper bounds on the probability cost for the random walk to stay in a bad region.
\begin{lemma}
\label{rstepsnt}
Let $\mb$ be as in \eqref{eq:lambda*lb}. There exist constants $C>0, b_2 \in (0,1)$ such that with $\widehat\P$-probability tending to one as $n \to \infty$, for all $x\in\Z^d$ and $t \geq \rad^2/2$,
	\begin{align}
		\Pr^x(\tau_{\fregion^c \cup \ob \cup B(\ct{\fregion},b_2\rad)} > t) &\leq Ce^{-100\mb \rad^{-2}t}\,.\label{eq:badprob-1}
	\end{align}
\end{lemma}
\begin{proof}
\eqref{eq:badprob-1} can be proved by a straightforward adaptation of the proof of \cite[Lemma 6.1]{DX18}, using the fact that $|B(\ct{\fregion},\rad) \setminus B(\ct{\fregion},b_2\rad)|\leq C(1-b_2)\rad^d$ with $b_2\in (0,1)$ chosen sufficiently close to $1$.
\end{proof}

\subsection{Upper and lower bounds on transition probabilities}
\label{sec:transition probability}
We will prove Lemma \ref{pt-loc}, Lemma \ref{loclemma2}, and Corollary \ref{loclemma} in this section. We have proved in \eqref{eq:thm-ballopen} that $B(\ct{\fregion},(1 - \rad^{-\kappa})\rad)$ is open with $\widehat \P$-probability tending to one as $n\to\infty$. This immediately leads to the following lower bound on the eigenfunction $\Phi_\fregion$ in the interior of the ball $B(\ct{\fregion},\rad)$.
\begin{lemma}
\label{phi-dist-bd}
 There exists a constant $c>0$ such that the following holds with $\widehat\P$-probability tending to one as $n \to \infty$: For all $x \in B(\ct{\fregion},(1 - 2 \rad^{-\kappa})\rad)$,
\begin{equation}
\label{eq:phi-dist-bd}
	\Phi_\fregion(x) \geq c \rad^{-d-1} \cdot \mathrm{dist}(x,B(\ct{\fregion},\rad)^c)\,.
\end{equation}
\end{lemma}
\begin{proof}
Note that since
$\langle  \Phi_\fregion, (P|_\fregion)^t \11_{x}\rangle= \lambda_\fregion^t\Phi_\fregion(x)$ for all $x$
 and $\lambda_\fregion \leq 1$, we have
\begin{equation}
\label{eq:3143}
\Phi_\fregion(x) \,\geq\, \frac{1}{2}\sum_{i =  \rad^2 , \rad^2 + 1} \sum_{y \in \fregion} \Phi_\fregion(y)\Pr^{y}(S_i = x, \tau_{\fregion^c} >i) \,,
\end{equation}
where we sum over two values of $i$ because the walk has period $2$.
Since $B(\ct{\fregion},\rad - \rad^{1-\kappa})$ is open with $\widehat \P$-probability tending to one as $n\to\infty$ by \eqref{eq:thm-ballopen}, we know that $B(\ct{\fregion},\rad - \rad^{1-\kappa}) \subset \fregion$. Hence for all $y \in B(\ct{\fregion},b_1\rad)$ ($b_1 = b_1(d,p)$ is chosen in Lemma \ref{eigenfunction value in the center}), \cite[Proposition 6.9.4]{LL10} yields for any $x \in B(\ct{\fregion},(1 - 2 \rad^{-\kappa})\rad)$,
\begin{equation}
\label{eq:3144}
	\sum_{i =  \rad^2 , \rad^2 + 1}\Pr^{y}(S_{i} = x, \tau_{\fregion^c} > i) \geq c\,\cdot\,\mathrm{dist}(x,\partial B(\ct{\fregion},(1 - 2 \rad^{-\kappa})\rad))\rad^{-d-1}\,,
\end{equation}
where $c$ is a constant depending only on $(d,p)$. Substituting \eqref{eq:3144} into \eqref{eq:3143} for $y \in B(\ct{\fregion},b_1\rad)$ and then
using Lemma \ref{eigenfunction value in the center} gives \eqref{eq:phi-dist-bd}.
\end{proof}

\begin{proof}[\bf Proof of Lemma \ref{pt-loc}]
It is proved in \cite[(6.15)]{DX18} that \eqref{eq:pt-loc} holds if $\Phi_\fregion(z) \geq c\epsilon \rad^{-d}$ for $z\in B(\ct{\fregion},(1-\epsilon)\rad)$ in addition to the assumption~\eqref{eq:ass-pt-loc} in Lemma~\ref{pt-loc}. This additional assumption is verified by Lemma~\ref{phi-dist-bd}.
\end{proof}

That the ball $B(\ct{\fregion},(1 - \rad^{-\kappa})\rad)$ is open implies that if the random walk starts from the interior ball $B(\ct{\fregion},b_2\rad)$ ($b_2$ defined in Lemma \ref{rstepsnt}), then in the next $\rad^2$ steps, all points in $B(\ct{\fregion},(1-\epsilon)\rad)$ can be reached with comparable probability. Lemma \ref{loclemma2} will follow from the following lemma, which says that conditioned to stay in $\fregion$, the random walk has a positive probability of visiting the interior of $B(\ct{\fregion},b_2\rad)$ in any given time interval of length $C\rad^2$.

\begin{lemma}
\label{Cr2}
There exist constants $C_9,c>0$ such that with $\widehat\P$-probability tending to one as $n \to \infty$, the following holds: For any $u$ and $t$ satisfying
\begin{equation}
\label{eq:ass-Cr2}
	\text{either } u \in B(\ct{\fregion},b_2\rad) \text{{ and }}t \geq 0 \quad \text{   or   }\quad u \in \fregion \text{ and } t \geq \rad^{C_4}\,,
\end{equation}
we have that for all $m \geq t$, 
	\begin{align}
	\label{eq:prblb-u}
	& \Pr^u(\tau_{\fregion^c} \geq m) \geq c \lambda_\fregion^{m-t}\Pr^u(\tau_{\fregion^c} \geq t)\,,\\
	\label{eq:u-x-t-step3}
		&\Pr^u(S_{[t - C_9 \rad^2,t]} \cap B(\ct{\fregion},b_2\rad) \not = \varnothing \mid \tau_{\fregion^c} > m) \geq c \,.
	\end{align}
\end{lemma}

The second case in~\eqref{eq:ass-Cr2} is harder to deal with since the random walk may start far away from $B(\ct{\fregion},b_2\rad)$. However, it can be reduced to the first case by using the following lemma, which guarantees that conditioned on staying inside $\fregion$, the random walk starting in $\fregion$ reaches $B(\ct{\fregion},b_2\rad)$ before time $\rad^{C_4}$.
\begin{lemma}
\label{hit tOmega-as}
Let $C_3>0$ be as in Lemma \ref{pt-loc}. There exist $c>0$ and $C_4$ with $C_4 > C_3>0$ such that the following holds with $\widehat\P$-probability tending to one as $n \to \infty$: For all $u \in \fregion$ and $t \geq \rad^{C_4}$,
\begin{equation}
  \Pr^u( S_{[0,t]} \cap B(\ct{\fregion},b_2\rad) = \varnothing \mid  \tau_{\fregion^c}>t) \leq e^{-c \rad^{-2}t}\,.
\end{equation}
\end{lemma}
\begin{proof}
Lemma \ref{rstepsnt} implies
$ \Pr^u( S_{[0,t]} \subseteq \fregion \setminus B(\ct{\fregion},b_2\rad)) \leq C\exp({-100\mb t \rad^{-2}})$. Comparing it with \cite[(5.4)]{DX18} (which implies that $\Pr^u(\tau_{\fregion^c} > t)$ is bounded from below by $\exp(-2\mu_B\rad^{-2}t -(\log n)^C)$), and choosing a sufficiently large $C_4$ yields the desired result.
\end{proof}

\begin{proof}[\bf{Proof of Lemma \ref{Cr2}}]
We first prove the lemma when $t=m$; more precisely, there exists a constant $C_9 = C_9(d,p)$ such that for $t$ and $u$ satisfying either condition in \eqref{eq:ass-Cr2},
\begin{equation}
\label{eq:endpoint-step2}
    \Pr^u(S_{[t - C_9 \rad^2,t]} \cap B(\ct{\fregion},b_2\rad) = \varnothing \mid \tau_{\fregion^c}>t) \leq 1/100\,.
  \end{equation}
We will prove this by considering the last visit to $B(\ct{\fregion},b_2\rad)$, and for the rest of the time comparing the survival probability for the random walk in $\fregion \setminus B(\ct{\fregion},b_2\rad)$ to the survival probability in the whole region $\fregion$ with starting point in $B(\ct{\fregion},b_2\rad)$.

To reduce the entropy resulting from the many possible last visit times, we chop time into small windows of length $ \rad^2$ and let $$N_{\text{exit}} := \sup\{ k \in \N : S_{[t - k  \rad^2 +1,t]} \cap  B(\ct{\fregion},b_2\rad) = \varnothing\}\,.$$
Since $N_{\text{exit}}=\infty$ is equivalent to no visit to $B(\ct{\fregion},b_2\rad)$, we always have $N_{\text{exit}} < \infty$ in the first case in \eqref{eq:ass-Cr2}. In the second case, we see from Lemma \ref{hit tOmega-as} that
\begin{equation}
	\Pr^u(N_{\text{exit}} =\infty \mid \tau_{\fregion^c}>t) \leq e^{-c \rad^{-2}t}\,.
\end{equation}
Now, we claim that for large $C_9$,
  \begin{equation}
  \label{eq:Nexit = k}
    \Pr^u( N_{\text{exit}} \geq C_9 \mid \tau_{\fregion^c}>t) \leq 1/100\,,
  \end{equation}
  which then implies \eqref{eq:endpoint-step2}. It remains to verify \eqref{eq:Nexit = k}. To this end, we define stopping times $T_k = \inf \{ j \geq t - (k+1)  \rad^2+1: S_j \in B(\ct{\fregion},b_2\rad)\}$ for $k \geq 0$. Since on the event $\{N_{\text{exit}} = k\}$, we have $ k  \rad^2\leq t-T_k \leq (k+1)  \rad^2\,$, by the strong Markov property, $\Pr^u( N_{\text{exit}} = k, \tau_{\fregion^c}>t)$ equals
\begin{align}\label{eq-Nexit}
 \Ex^u[\11_{\tau_{\fregion^c}>T_k, T_k <t - k\rad^2} \Pr^{S_{T_k}}(\tau_{\fregion^c}>t - T_k,S_{[t-T_k - k \rad^2 +1,t-T_k]}\cap  B(\ct{\fregion},b_2\rad) = \varnothing)].
\end{align}
Now we consider all $x\in B(\ct{\fregion},b_2\rad)$ and $t- (k+1)\rad^2 \leq m \leq t - k\rad^2$ (which include all $(x, m)$ such that $(S_{T_k}, T_k) = (x, m)$ occurs with non-zero probability). On one hand, Lemma \ref{rstepsnt} implies
\begin{equation}
\label{eq:eptloc-bad}
  \Pr^{x}(\tau_{\fregion^c}>t - m,S_{[t-m - k \rad^2 +1,t-m]}\cap  B(\ct{\fregion},b_2\rad) = \varnothing) \leq C\exp( - 99\mb k)\,.
\end{equation}
On the other hand, by \eqref{eq:t-in-U-lb}, Lemmas \ref{lambda*}, \ref{phi-dist-bd}, and $\lambda_\fregion > \lambda_*$,
\begin{equation}
\label{eq:eptloc-gd}
  \Pr^x(\tau_{\fregion^c}>t-m) \geq c  \lambda_{\fregion}^{t-m} \geq c  \exp\{-2\mb  k\}\,.
\end{equation}
Combining the preceding two inequalities and \eqref{eq-Nexit} yields that for sufficiently large $k$,
 \begin{align*}
  \Pr^u( N_{\text{exit}} = k, \tau_{\fregion^c}>t)
  \leq &e^{-50\mb k} \Ex^u[\11_{\tau_{\fregion^c}>T_k} \Pr^{S_{T_k}}(\tau_{\fregion^c}>t - T_k)]= e^{-50\mb k} \Pr^u( \tau_{\fregion^c}>t)\,.
\end{align*}
We complete the proof of \eqref{eq:Nexit = k} by summing over all $k\geq C_9$ chosen sufficiently large.

Next, we prove \eqref{eq:prblb-u}. If we define stopping time $T_\star := \inf \{j \geq t - C_9 \rad^2: S_j \in B(\ct{\fregion},b_2\rad) \}$, then by {the strong Markov property at $T_\star$,} \eqref{eq:t-in-U-lb} and Lemma \ref{phi-dist-bd},
\begin{equation}
\label{eq:atspppp=1}
\begin{split}
  	\Pr^u(\tau_{\fregion^c} \geq m) &\geq \Pr^u(S_{[t - C_9 \rad^2,t]} \cap B(\ct{\fregion},b_2\rad) \not = \varnothing, \tau_{\fregion^c} > m)\\ &\geq  \Ex^u \big[\11_{T_\star < t, \tau_{\fregion^c} > T_\star} \cdot  c \lambda_\fregion^{m-T_\star}\big]\,\,.
\end{split}
\end{equation}
 Since $m-T_\star \leq m-t + C_9 \rad^2$ and $\lambda_\fregion > \lambda_* \geq 1 - \mb \rad^{-2} - C_* \rad^{-3}$ (see Lemma \ref{lambda*}), this is further bounded from below by
\begin{equation}
\label{eq:atspppp=2}
	c  \lambda_\fregion^{m-t + C_9\rad^2} \Pr^u(T_\star < t, \tau_{\fregion^c} > t) \geq c  \lambda_\fregion^{m-t}\Pr^u (\tau_{\fregion^c} >t)\,,
\end{equation}
where in the last inequality, we used \eqref{eq:endpoint-step2}. This gives \eqref{eq:prblb-u}.

Finally, we prove \eqref{eq:u-x-t-step3}. First note that by the Markov property at time $t$ and \eqref{eq:t-in-U-ub},
\begin{equation}
\label{eq:1286}
		\Pr^u(\tau_{\fregion^c} > m,S_m \in \hbn)\leq C \lambda_\fregion^{m-t}\Pr^u (\tau_{\fregion^c} >t)\,.
\end{equation}
Then by \eqref{eq:atspppp=1} and \eqref{eq:atspppp=2}, this is less than
$$ C\Pr^u(S_{[t - C_9 \rad^2,t]} \cap B(\ct{\fregion},b_2\rad) \not = \varnothing, \tau_{\fregion^c} > m)\,.$$
Combining this and Lemma \ref{pt-loc} yields \eqref{eq:u-x-t-step3}.
\end{proof}

\begin{proof}[\bf Proof of Lemma \ref{hit ball}]
Set $t = \lceil \rad^{C_4} \rceil$. By the Markov property at time $t$ and \eqref{eq:t-in-U-ub},
\begin{align*}
	\Pr^u( \tau_{B(\ct{\fregion},b_2\rad)} > t, \tau_{\fregion^c} > m,S_m \in \hbn) &\leq C\Pr^u( \tau_{B(\ct{\fregion},b_2\rad)} > t, \tau_{\fregion^c} > t)\lambda_\fregion^{m-t} \\
	&= C e^{-c t\rad^{-2}} \cdot \Pr^u(\tau_{\fregion^c} > t \big)\lambda_\fregion^{m-t}\,,
\end{align*}
where in the last step, we used Lemma \ref{hit tOmega-as}. Combined with the lower bound of $\Pr^u( \tau_{\fregion^c} > m)$ given by \eqref{eq:prblb-u}, it yields
$$ \Pr^u( \tau_{B(\ct{\fregion},b_2\rad)} > t,S_m \in \hbn \mid \tau_{\fregion^c} > m) \leq C e^{-c \rad^{C_4-2}}\,.$$
Combining it with Lemma \ref{pt-loc} gives \eqref{eq:hitball}.
\end{proof}

\begin{proof}[\bf Proof of Lemma \ref{loclemma2}]By adjusting the constant factor $c$ in \eqref{eq:loclemma2-1}, we may assume $\epsilon < 1- b_2$.
Let $u\in B(\ct{\fregion},b_2\rad)$. We first prove that for $t \geq \rad^2$,
\begin{equation}
\label{eq:416}
	\min_{\substack{x \in B(\ct{\fregion},(1 - \epsilon)\rad)\\ |x-u|_1 + t\ is \ even}}\Pr^u(\tau_{\fregion^c} > t, S_t = x) \geq c \epsilon \max_{y \in \fregion} \Pr^u(\tau_{\fregion^c} > t, S_t = y)\,.
\end{equation}
To this end, we define stopping time $T_\star = \inf \{j \geq t - (C_9 + 1) \rad^2: S_j \in B(\ct{\fregion},b_2\rad) \}$(with $T_\star = 0$ for $t \leq (C_9 + 1) \rad^2$). Then by \eqref{eq:u-x-t-step3},
\begin{equation}
\label{eq:5.24}
\Pr^u(T_\star \leq t- \rad^2 \mid \tau_{\fregion^c} > t - \rad^2) \geq c\,.	
\end{equation}
Then for all $x \in  B(\ct{\fregion},(1 -\epsilon)\rad)$ such that $|x-u|_1 + t$ is even,
\begin{equation}
\label{eq:5.23}
	\Pr^u(\tau_{\fregion^c} > t, S_t = x) \geq \Ex^u \big[\11_{\tau_{\fregion^c}> T_\star,T_\star \leq t- \rad^2} \Pr^{S_{T_\star}}(\tau_{\fregion^c} > t - T_\star, S_{t - T_\star} = x) \big]
\end{equation}
Since $B(\ct{\fregion},(1 - \rad^{-\kappa})\rad) \subset \fregion$ with $\widehat \P$-probability tending to one as $n\to\infty$ by Theorem \ref{thm:Ball confinement}, it follows from \cite[Proposition 6.9.4]{LL10} that uniformly in $x \in  B(\ct{\fregion},(1 -\epsilon)\rad), y \in B(\ct{\fregion},b_2 \rad)$ and $\rad^2 \leq k \leq (C_9 + 1) \rad^2$ such that $|x-u|_1 + k$ is even, we have
\begin{equation}
	\Pr^y(\tau_{\fregion^c} > k, S_k = x) \geq c \epsilon \rad^{-d}\,.
\end{equation}
Substituting this bounds into \eqref{eq:5.23} yields
\begin{align*}
	\Pr^u(\tau_{\fregion^c} > t, S_t = x)& \geq c \epsilon\rad^{-d} \Pr^u(\tau_{\fregion^c}> T_\star,T_\star \leq t- \rad^2) \\
	& \geq c \epsilon\rad^{-d} \Pr^u(\tau_{\fregion^c}> t- \rad^2,T_\star \leq t- \rad^2)\\
	&\geq c \epsilon\rad^{-d} \Pr^u(\tau_{\fregion^c}> t - \rad^2)\,,
\end{align*}
where we used \eqref{eq:5.24}.
On the other hand, for all $y \in \fregion$,
\begin{equation}
\begin{split}
 \Pr^u(\tau_{\fregion^c} > t, S_t = y) &= \Ex^u \big[\11_{\tau_{\fregion^c}> t - \rad^2}\Pr^{S_{ t - \rad^2}}(\tau_{\fregion^c} > \rad^2,S_{\rad^2} = y)\big]\\ &\leq C \rad^{-d} \Pr^u(\tau_{\fregion^c}> t - \rad^2)\,.
\end{split}
\end{equation}
Combining the two preceding bounds give \eqref{eq:416}.

We now prove \eqref{eq:loclemma2-2} and \eqref{eq:loclemma2-1}. Combining Lemmas \ref{stayinU} and \ref{phi-dist-bd} gives that for $m-t \geq 0$,
\begin{equation}
\label{eq:4.17}
	\min_{x \in B(\ct{\fregion},(1-\cec)\rad)}\Pr^x(\tau_{\fregion^c} > {m-t}) \geq c \epsilon \max_{y \in \fregion} \Pr^y(\tau_{\fregion^c} > {m-t}, S_{m-t} \in \hbn)\,.
\end{equation}
Multiplying each side of \eqref{eq:4.17} with that of \eqref{eq:416} and using the Markov property at time $m-t$, we obtain\begin{equation*}
	\min_{\substack{x \in B(\ct{\fregion},(1-\cec)\rad),\\ |x-u|_1 + t\text{ is even}}}\Pr^{u}(S_t = x ,  \tau_{\fregion^c} > m) \geq c\epsilon^2 \max_{y \in \fregion}  \Pr^{u}(S_t = y , S_m \in \hbn, \tau_{\fregion^c} > m)\,.
\end{equation*}
Lemma \ref{pt-loc} implies
$$\Pr^u( S_m \in \hbn \mid \tau_{\fregion^c} > m) \geq  1 - \exp({-\rad^c}) \quad\text{and}\quad \max_{y \in \fregion}  \Pr^{u}(S_t = y \mid \tau_{\fregion^c} > m)  \geq c \rad^{-d}\,.$$ Therefore,
\begin{equation}
\label{eq:54831}
	\min_{\substack{x \in B(\ct{\fregion},(1-\cec)\rad),\\ |x-u|_1 + t\text{ is even}}}\Pr^{u}(S_t = x \mid \tau_{\fregion^c} > m) \geq c\epsilon^2 \max_{y \in \fregion}  \Pr^{u}(S_t = y \mid \tau_{\fregion^c} > m) -\exp({-\rad^c})\,,
\end{equation}
then \eqref{eq:loclemma2-2} follows. In addition, \eqref{eq:54831} implies
\begin{equation}
\begin{split}
  1 &\geq \Pr^{u}(S_t \in B(\ct{\fregion}, \rad/2 )\mid \tau_{\fregion^c} > m)\\ & \geq c \rad^{d}\max_{y \in \fregion}  \Pr^{u}(S_t = y \mid \tau_{\fregion^c} > m) -C \rad^d\exp({-\rad^c})\,,
\end{split}
\end{equation}
which yields \eqref{eq:loclemma2-1}.
\end{proof}

\begin{proof}[\bf Proof of Corollary \ref{loclemma}]
We first consider the case when $t \geq \rad^2$. Since  $| \hbn \setminus B(\ct{\fregion},(1 - 2\rad^{-\kappa}) \rad))| \leq \rad^{d-c}$ for some constant $c \in (0,1)$, by \eqref{eq:loclemma2-2},
\begin{equation}
\Pr^{u}(S_t \in \hbn \setminus B(\ct{\fregion},(1 - 2\rad^{-\kappa}) \rad) \mid \tau_{\fregion^c} > m)  \leq C\rad^{-c}\,.
\end{equation}
Combined with Lemma \ref{pt-loc}, it yields \eqref{eq:99999}.

Now, we consider the case $t \leq \rad^2$. For $u \in B(\ct{\fregion},b_2\rad)$ , since $$\mathrm{dist}(u, \hbn \setminus B(\ct{\fregion},(1 -  2\rad^{-\kappa}) \rad)) \geq (1 - b_2) \rad/2\,,$$ by a union bound and the local limit theorem, we have that for any $t \geq 0$,
\begin{equation*}
	\Pr^u(S_t \in \hbn \setminus B(\ct{\fregion},(1 - 2\rad^{-\kappa}) \rad)) \leq C|\hbn \setminus B(\ct{\fregion},(1 - 2\rad^{-\kappa})\rad)| \rad^{-d} \leq \rad^{-c}\,.
\end{equation*}
Then by the Markov property at time $t$ and \eqref{eq:t-in-U-ub},
\begin{equation}
\label{eq:187235-1}
	\Pr^u(S_t \in \hbn \setminus B(\ct{\fregion},(1 - 2\rad^{-\kappa}) \rad), S_m \in \hbn, \tau_{\fregion^c} > m) \leq \rad^{-c} \cdot \lambda_\fregion^{m-t}\,.
\end{equation}
On the other hand, combining Lemma \ref{eq:t-in-U-lb} and Lemma \ref{phi-dist-bd} gives
\begin{equation}
\label{eq:187235-2}
\Pr^u(\tau_{\fregion^c} > m) \geq c \lambda_\fregion^{m}\,.
\end{equation}
Since Lemma \ref{lambda*} yields $\lambda_\fregion^{-t} \leq C$ for $t\le \rad^2$, combining \eqref{eq:187235-1}, \eqref{eq:187235-2}, and  Lemma \ref{pt-loc} gives \eqref{eq:99999}.
\end{proof}

\subsection{Distribution of the random walk}
\label{sec:distribution}
This section is devoted to the proof of Lemma \ref{distribution}. We will prove Lemma \ref{distribution} by the eigenfunction expansion of $P|_\qbn$. Loosely speaking, if we know that the spectral gap is larger than $c\rad^{-2}$, then after time much longer than $\rad^2$, the principal eigenfunction term should dominate all the other terms. However, there is an issue caused by the periodicity of the random walk, that is, there is a negative eigenvalue with the same modulus as the principal eigenvalue. In order to circumvent this issue, we will deal with even and odd times and sites separately. This corresponds to dealing with $(P|_\qbn)^2$, instead of $P|_\qbn$, and we will prove necessary estimates for the corresponding eigenvalues and eigenfunctions in Appendix~\ref{app:eigen}.

Let $\eva{i}{M}$ denote the $i$-th largest eigenvalue of the matrix $M$ and let $\eve{i}{M}$ denote the corresponding $\ell^1$-normalized eigenvector. For any vector $\eta$ indexed by $\Z^d$, we let $\eta_{\be}$ and $\eta_{\bo}$ be $\eta$ restricted to even and odd sites in $\Z^d$, respectively. Also, for any matrix $M$ indexed by $\Z^d \times \Z^d$, we let $M_{\be}$ and $M_{\bo}$ be $M$ with both coordinates restricted to even and odd sites, respectively.

\begin{proof}[\bf Proof of Lemma \ref{distribution}]
Denote $Q = P|_\qbn$ and $\eta = \eve{1}{Q}$ to simplify the notation.  Since $B(\ct{\fregion}, \rad - \rad^{1-\kappa})$ is open with $\widehat \P$-probability tending to one as $n\to\infty$ by \eqref{eq:thm-ballopen}, we have
 $$ B(\ct{\fregion}, \rad - \rad^{1-\kappa}) \subset \qbn \subset B(\ct{\fregion}, \rad + \rad^{1-c_1})\,.$$
Hence the assumption \eqref{eq:B-assmp} in Lemma \ref{qbnproperty} holds. By the eigendecomposition of $Q^2_{\be}$ and \eqref{eq:parity-2}, we have that for any even site $ v \in \Z^d$ and $m \geq 0$,
\begin{equation}
\label{eq:eign-even}
\begin{split}
		\left|\11_{v}^\tran (Q^2_{\be})^m - \eva{1}{Q^2_{\be}}^m  \frac{\eta_{\be}(v)}{|\eta_{\be}|_2^2} \eta_{\be} \right|_1& \leq  \sum_{i \geq 2}\eva{i}{Q^2_{\be}}^m {\frac{\langle \Phi_i(Q)_{\be},\11_v\rangle}{|\Phi_i(Q)_{\be}|_2^2}|\Phi_i(Q)_{\be}|_1}
\end{split}
\end{equation}
Since the dimension of the matrix $Q$ is at most $|\hbn|$, and \eqref{eq:qbn-1} implies $$|\Phi_i(Q)_{\be}|^2_2 = \frac{1}{2} |\Phi_i(Q)|^2_2 \geq \frac{|\Phi_i(Q)|^2_1}{2 \,\mathrm{dim}(Q)}\,,$$ we can further bound the right hand side of \eqref{eq:eign-even} from above by
\begin{equation}
	 2\sum_{i \geq 2}\eva{i}{Q^2_{\be}}^m |\qbn|\leq 2\eva{1}{Q^2_{\be}}^m |\qbn|^2e^{-c \rad^{-2}m} = \lambda_\qbn^{2m} |\qbn|^2e^{-c \rad^{-2}m}\,,
\end{equation}
where we used \eqref{eq:qbn-2} and $\lambda_\qbn = \eva{1}{\qbn}$ as defined before Lemma \ref{lambda*}.
Then since $|x^\tran Q|_1 \leq |x|_1$ for all $x$, $\11_{v}^\tran (Q^2_{\be})^t = \11_{v}^\tran Q^{2t}$, $Q\eta_{\be}= \lambda_\qbn\eta_{\bo}$, we have
\begin{equation}
\label{eq:eign-odd}
	\left|\11_{v}^\tran Q^{2m+1} - \lambda_\qbn^{2m+1}  \frac{\eta_{\be}(v)}{|\eta_{\be}|_2^2} \eta_{\bo} \right|_1\leq \lambda_\qbn^{2m+1} |\qbn|^2e^{-c \rad^{-2}m}\,.
\end{equation}
Fix $C'$ to be some large constant to be determined. Combining \eqref{eq:eign-even} and \eqref{eq:eign-odd} with \eqref{eq:qbn-1}, we get that for any even site $v \in \Z^d$, $m \geq C' \rad^{2}\log \log n$ and $x$ such that $|x-v|_1 + m$ is even,
\begin{equation}
\label{eq:qbnprob-1}
\left| \Pr^v(S_{m} = x,\tau_{\qbn^c}>m) - 2|\eta|_2^{-2}\lambda_\qbn^{m}  \eta(v)\eta(x)\right| \leq \lambda_\qbn^{m} (\log n)^{-c C'}\,.
\end{equation}
Similarly, this also holds for odd site $v \in \Z^d$.
Summing \eqref{eq:qbnprob-1} over $x$ and using \eqref{eq:qbn-1}, we get
\begin{equation}
\label{eq:qbnprob-2}
	\Pr^v(\tau_{\qbn^c}>m) = |\eta|_2^{-2}\lambda_\qbn^{m} \eta(v)(1  + O(\rad^{-2}))\,.
\end{equation}
Hence
\begin{equation}
\left| \Pr^v(S_{m} = x \mid \tau_{\qbn^c}>m) - 2\eta(x)\right| \leq C|\eta|_2^2\eta(v)^{-1}(\log n)^{-c C'} + C\eta(x) \rad^{-2}\,.
\end{equation}
Since $v \in B(\ct{\fregion},(1 - 2\rad^{-\kappa}) \rad)$, \eqref{eq:qbn-3} yields $\eta(v) \geq c\rad^{-d-\kappa}$. Also \eqref{eq:qbn-0} yields $|\eta|^2_2 \leq C\rad^{-d}$.  Now, choosing a sufficiently large $C'$ %, summing the preceding equation over $x$ such that $|x-v|_1 + m$ is even,
and applying \eqref{eq:A2-3} to replace $\eta(x)$ by $\rad^{-d}\phi_1(\frac{x - \ct{\fregion}}{\rad})$, we get \eqref{eq:endpt-m-t}.

On the other hand, define for $m \geq 0$, $u,v \in \Z^d$, $$q_m(u,v): = 2|\eta|_2^{-2}\lambda_\qbn^{m}  \eta(u)\eta(v)\,.$$ Then combining \eqref{eq:qbnprob-1}, \eqref{eq:qbnprob-2} and \eqref{eq:qbn-0} yields that for any $v,x,y\in \Z^d$ and $m,t \geq C' \rad^{2}\log \log n$ such that both $|x-v|_1 + m$ and $|y-v|_1 + m+t$ are even,
\begin{align*}
&\left| \Pr^v(S_{m} = x,S_{m+t} =y,  \tau_{\qbn^c}>m + t) - q_m(v,x)q_t(x,y)\right|\\
 = &	\left| \Pr^v(S_m = x,  \tau_{\qbn^c}>m) \Pr^x(S_t =y,  \tau_{\qbn^c}>t)- q_m(v,x)q_t(x,y)\right|	\\
 \leq & \lambda_\qbn^{m+t} (\log n)^{-2c C'} + q_m(v,x)\lambda_\qbn^{t} (\log n)^{-c C'} + q_t(x,y)\lambda_\qbn^{m} (\log n)^{-c C'}\\
 \leq & 5\lambda_\qbn^{m+t} (\log n)^{-c C'} \,.
\end{align*}
Combined with \eqref{eq:qbnprob-1}, this yields
\begin{equation*}
	\left| \Pr^v(S_{m} = x \mid S_{m+t} =y,  \tau_{\qbn^c}>m + t) - 2 |\eta|_2^{-2}\eta(x)^2\right| \leq(\log n)^{-cC'+C} \frac{|\eta|_2^{2}}{\eta(v)\eta(y)}.
\end{equation*}
Since $v,y \in B(\ct{\fregion},(1 - 2\rad^{-\kappa}) \rad)$, \eqref{eq:qbn-3} yields $\eta(v),\eta(y) \geq c\rad^{-d-\kappa}$. Also \eqref{eq:qbn-0} yields $|\eta|^2_2 \leq C\rad^{-d}$.  Now, choosing a sufficiently large $C'$ %, summing the preceding equation over $x$ such that $|x-v|_1 + m$ is even,
and applying \eqref{eq:A2-2} to replace $|\eta|_2^{-2}\eta(x)$ by $\rad^{-d/2}\phi_2(\frac{x - \ct{\fregion}}{\rad})$, we get \eqref{eq:bulk-m-t}.
\end{proof}

\appendix
\section{Estimates for eigenvalues and eigenfunctions}
\label{app:eigen}
This section collects some basic estimates for eigenfunctions and eigenvalues used in the proof. For finite $A \subset \Z^d$, we let $\lambda_{A}$ denote the principal (largest) eigenvalue of $P|_{A}$, which is the transition matrix of the simple symmetric random walk on $\Z^d$ killed upon exiting $A$, and let $\Phi_A$ be the $\ell^1$-normalized principal eigenfunction of $P|_{A}$. The following lemma bounds the $\ell_\infty$-norm of the eigenfunction $\Phi_D$ in a finite domain $D \subset \Z^d$ in terms of the eigenvalue $\lambda_D$.
\begin{lemma}
\label{Philinfbd}
	There exists a constant $C>0$ such that $ |\Phi_D|_\infty \leq C(1 - \lambda_D)^{d/2}$ for all finite $D \subset \Z^d$.
\end{lemma}
\begin{proof}
By $P|_{D}\Phi_D = \lambda_{D} \Phi_D$, we have $\sum_{u:u \sim v}(2d)^{-1}\Phi_D(u) = \lambda_{D} \Phi_D(v)$. Then it follows from the Markov property that $(\lambda_{D}^{-t \wedge \tau_{D^c}}\Phi_D(S_{t\wedge \tau_{D^c}}))_{t\geq 0}$ is a martingale.

 Let $l = (1 - \lambda_D)^{-1/2} $. By the optional sampling theorem and the local limit theorem,
\begin{equation*}
\begin{split}
 	\Phi_D(v) = \Ex^v[\lambda_{D}^{-{\lfloor l^2 \wedge \tau_{D^c}}\rfloor}\Phi_D(S_{{\lfloor l^2 \rfloor}\wedge\tau_{D^c}})] &\leq  \lambda_{D}^{-l^2}  \sum_{u} \Pr^v(S_{\lfloor l^2 \rfloor} = u) \Phi_D(u) + \Pr^v(\tau_{D^c} \leq \lfloor l^2 \rfloor) \cdot 0\\
 	&\leq C l^{-d} \quad \text{uniformly in $v \in D$.} \qedhere
\end{split}
 \end{equation*}
 \end{proof}

 Let $\eva{i}{M}$ denote the $i$-th largest eigenvalue of matrix $M$ and $\eve{i}{M}$ denote the corresponding $\ell^1$-normalized eigenvector. The following lemma says that if a large domain in $\Z^d$ is close to a ball, then the first eigenvalue and eigenfunction of this domain are also close to that of the ball.
\begin{lemma}
\label{eigenfunction}
	Suppose $B(0,(1 - \epsilon) t) \subset \mathcal B \subset B(0,(1 + \epsilon)t)$ where $\epsilon$ is smaller than some constant depending only on $d$. Then there exist constants $C,c>0$ such that
	\begin{align}
		&\label{eq:A2-1}\eva{1}{P|_{\mathcal B}} - \eva{2}{P|_{\mathcal B}} \geq c t^{-2}\,,\\
		&\label{eq:A2-2}\Big|\frac{\eve{1}{\mathcal B}^2}{|\Phi_1(\mathcal B)|_2^2} - \phi^2_{2,t}\Big|_1  \leq  C(\sqrt{\epsilon} + t^{-1/2})\,,\\
		&\label{eq:A2-3}|\eve{1}{\mathcal B} - \phi_{1,t}|_1  \leq  C(\sqrt{\epsilon} + t^{-1/2})\,,
	\end{align}
where $\phi_1$ and $\phi_2$ are respectively the $L^1$ and $L^2$-normalized first eigenfunction of the Dirichlet-Laplacian of the unit ball in $\R^d$, and $\phi_{2,t}(\cdot) = t^{-d/2}\phi_2(\cdot/t)$, $\phi_{1,t}(\cdot) = t^{-d}\phi_1(\cdot/t)$.
\end{lemma}

\begin{proof}
First, we see that for $i = 1,2$, by \cite[(3.27) and (6.11)]{Weinberger58}
\begin{equation}
\label{eq:eigen-ball-ref}
	\eva{i}{P|_{B(0,t)}} = 1 - t^{-2}\mu_i(\mathbf B) + O(t^{-3})\,,
\end{equation}
where $\mu_i(\mathbf B)$ is the $i$-th eigenvalue of the Dirichlet-Laplacian of the unite ball $\mathbf B \subset \R^d$.
The min-max theorem implies that $\eva{i}{B(0,(1 - \epsilon) t)} \leq \eva{i}{\mathcal B} \leq \eva{i}{B(0,(1 + \epsilon)t)}$. Hence for $i = 1,2$
\begin{equation*}
	\eva{i}{P|_{\mathcal B}} = 1 - t^{-2}\mu_i(\mathbf B)  + O(\epsilon t^{-2} + t^{-3})\,.
\end{equation*}
This implies the first assertion.

Let $\phi_{2,\mathcal B}$ be the $\ell^2$-normalized first eigenvector of $P|_{\mathcal B}$, let $\phi_2$ be the $L^2$-normalized first eigenfunction of the Dirichlet-Laplacian of the unit ball in $\R^d$, and define
\begin{equation*}
	\tilde \phi_{2,t}(x) := t^{d/2}\int_{ x/t+ [0,1/t]^d} \phi_2(y) \dif y, ~x \in \Z^d\,.
\end{equation*}
Then by \cite[(6.11)]{Weinberger58} and \cite[(1.5)]{Weinberger60},
 \begin{equation*}
	\Big|\phi_{2,\mathcal B} - \frac{\tilde \phi_{2,(1 - \epsilon) t}}{|\tilde \phi_{2,(1 - \epsilon) t}|_2}\Big|^2_2 \leq C \frac{\eva{1}{P|_{\mathcal B}} -  \eva{1}{P|_{B(0,(1 - \epsilon) t)}}}{\eva{1}{P|_{\mathcal B}} - \eva{2}{P|_{\mathcal B}}} = O(\epsilon + t^{-1})\,.
\end{equation*}
Since $\phi_2$ is continuously differentiable (see, for example,~\cite[Corollary 8.11]{GT01}),
\begin{equation*}
	|\tilde \phi_{2,(1 - \epsilon) t} - \phi_{2,t}|_\infty = O(t^{-d/2-1})\quad {and} \quad|\tilde \phi_{2,(1 - \epsilon)t}|^2_2 = 1 + O(t^{-1})\,.
\end{equation*}
Altogether, we have $\sum_{x \in \Z^d} (\phi_{2,\mathcal B}(x) -  t^{-d/2}\phi_2(x/t))^2  = O(\epsilon + t^{-1})$. The second and third assertion follow by combining this with the boundedness of $\phi_2$ and Lemma \ref{Philinfbd}.
\end{proof}
The following two lemmas are needed to deal with the periodicity of the simple random walk. In what follows, for any vector $\eta$ indexed by sites in $\Z^d$, we let $\eta_{\be}$ and $\eta_{\bo}$ be $\eta$ restricted to even and odd sites in $\Z^d$, respectively.
\begin{lemma}
\label{parity}
Let $Q = P|_A$ for some finite $A \subset \Z^d$. We denote by $Q^2_{\be}$ and $Q^2_{\bo}$ the transition matrix $Q^2$ restricted to even and odd sites in $\Z^d$,respectively. Then  $$\mathrm{rank}~Q^2_{\be}=\mathrm{rank}~Q^2_{\bo} = \mathrm{rank}~Q/2\,.$$
For $1 \leq i \leq \mathrm{rank}~Q/2$, we have
	\begin{equation}
	\label{eq:parity-1}
\eva{i}{Q}^2= \eva{i}{Q^2_{\bo}} = \eva{i}{Q^2_{\be}}\,,
	\end{equation}
and
\begin{equation}
	\label{eq:parity-2}
	\frac{\eve{i}{Q}_{\be}}{|\eve{i}{Q}_{\be}|_1}= \eve{i}{Q^2_{\be}}, \quad 	\frac{\eve{i}{Q}_{\bo}}{|\eve{i}{Q}_{\bo}|_1}= \eve{i}{Q^2_{\bo}}\,.
\end{equation}
Furthermore,
\begin{equation}
\label{eq:g-123}
\begin{split}
		&|\eve{i}{Q}_{\be}|_2=|\eve{i}{Q}_{\bo}|_2\,,\\
	&|\eva{i}{Q}| \leq \frac{|\eve{i}{Q}_{\be}|_1}{|\eve{i}{Q}_{\bo}|_1} \leq |\eva{i}{Q}|^{-1}\,.
\end{split}
\end{equation}
\end{lemma}
\begin{proof}
 Note that if $\lambda$ is an eigenvalue of $Q$ with eigenvector $\eta$,  then
\begin{equation}
\label{eq:333123}
	Q \eta_{\bo}= \lambda \eta_{\be}, \quad  Q \eta_{\be}= \lambda \eta_{\bo}\,,
\end{equation}
and hence $-\lambda$ is an eigenvalue of $Q$ with eigenvector $\eta_{\be} - \eta_{\bo}$. Furthermore, $\eta_{\bo}$ is in the null space of the matrix $Q^2_{\be}$, hence $\eta_{\be}$ is an eigenvector of $Q^2_{\be}$ associated with eigenvalue $\lambda$. Similarly, $\eta_{\bo}$ is an eigenvector of $Q^2_{\bo}$ associated with eigenvalue $\lambda$. Therefore, we conclude that the nonzero eigenvalues of $Q^2_{\be}$ (or $Q^2_{\bo}$) are exactly the square of positive eigenvalues of $Q$. The corresponding eigenfunctions can be found by restricting eigenfunctions of $Q$ to odd (or even) sites. Hence \eqref{eq:parity-1} and \eqref{eq:parity-2} follow. \eqref{eq:g-123} follows directly from \eqref{eq:333123}.
\end{proof}
\begin{lemma}
\label{qbnproperty}
Let $a \in (0,1)$ and $Q = P|_B$ where $B$ is a subset of $\Z^d$ that satisfies \begin{equation}
\label{eq:B-assmp}
	B(0, n - n^{1-a}) \subset B \subset B(0, n + n^{1-a})\,.
\end{equation}
Then there exist constants $C,c>0$ depending only on $(a,d)$ such that for sufficiently large $n$,
 \begin{align}
  \label{eq:qbn-0}
  &\eva{1}{Q} \geq 1 - C n^{-2}\,, \quad \quad \quad |\eve{1}{Q}|_\infty \leq C n^{-d}\\
 \label{eq:qbn-1}
 	&|\eve{1}{Q}_{\be}|_2 = |\eve{1}{Q}_{\bo}|_2\,, \quad \quad|\eve{1}{Q}_{\be}|_1, |\eve{1}{Q}_{\bo}|_1 = 1/2 + O(n^{-2})\\
 \label{eq:qbn-2}	&	\eva{1}{Q^2_{\be}} - \eva{2}{Q^2_{\be}} = \eva{1}{Q^2_{\bo}}- \eva{2}{Q^2_{\bo}}  \geq c n^{-2}\,.
	\end{align}
	For $x \in B(0,(1 - 2 n^{-c})n)$, we have
\begin{equation}
\label{eq:qbn-3}
	\eve{1}{Q}(x) \geq c n^{-d-1} \cdot \mathrm{dist}(x,B(0,n)^c)\,.
\end{equation}
\end{lemma}
\begin{proof}
First, \eqref{eq:eigen-ball-ref} gives $\eva{1}{Q} \geq 1 - C n^{-2}$. Combining with Lemma \ref{Philinfbd}, we get \eqref{eq:qbn-0}. Also, by \eqref{eq:A2-1}, we know that $
	\eva{1}{Q} - \eva{2}{Q} \geq c n^{-2}\,.$ Hence \eqref{eq:qbn-1} follows from \eqref{eq:g-123} and \eqref{eq:qbn-2} follows from \eqref{eq:parity-1}.

Next, we verify \eqref{eq:qbn-3}. We first see that $|\eve{1}{Q}|_\infty \leq C n^{-d}$ yields that for some constant $c' >0$,
$$ \sum_{x \in B(0,(1 - c')n)}\eve{1}{Q}(x) \geq 1/2. $$
Then the proof of Lemma \ref{phi-dist-bd} also works here.
\end{proof}

\section{Comparison of eigenvalues on nested domains}
In this section, we derive upper and lower bounds on the eigenvalue decrement after we remove a subset from the domain in Lemmas \ref{eigenvalue-difference} and \ref{RemoveB}, respectively. In particular, we are interested in the case when the subset being removed is very close to the boundary as in Lemma \ref{RemoveB}. More precisely, we show the following, where 
$$
\partial A := \{x\in A^c: |x-y|_1 = 1 \mbox{ for some } y\in A\}\,.
$$
We will follow the same notation as in Appendix \ref{app:eigen}.
\begin{lemma}
\label{eigenvalue-difference}
	Let $D_2 \subset D_1 \subset \Z^d$ be finite and $q = \sum_{x \in ( D_2 \cup \partial D_2)} \Phi_{D_1}^2(x)/|\Phi_{D_1}|_2^2$. Then
	\begin{equation}
		\lambda_{D_1} - \lambda_{D_1 \setminus D_2} \leq \frac{2 q}{1 - q}\,.
	\end{equation}
\end{lemma}
\begin{proof}
Let $\tilde \Phi_{D_1}(v) = \Phi_{D_1}(v)\11_{v \not \in D_2}$, then $\tilde \Phi_{D_1}$ is supported on $D_1 \setminus D_2$ and
$$ |\tilde \Phi_{D_1}|_2^2 \geq |\Phi_{D_1}|_2^2 (1 - q)\,.$$
For any adjacent $x, y \in \Z^d$ such that $(x,y) \not \in (\partial D_2^c \times \partial D_2) \cup (\partial D_2 \times \partial D_2^c)$, we have $$( \tilde \Phi_{D_1}(x) -  \tilde \Phi_{D_1}(y))^2 \leq (  \Phi_{D_1}(x) -   \Phi_{D_1}(y))^2\,.$$
Hence,
$$ \frac{1}{4d}\sum_{(x,y):|x-y|_1 = 1} \big [( \tilde \Phi_{D_1}(x) -  \tilde \Phi_{D_1}(y))^2 -  (  \Phi_{D_1}(x) -   \Phi_{D_1}(y))^2 \big] \leq \sum_{x \in \partial D_2} \Phi_{D_1}^2(x) \leq q |\Phi_{D_1}|_2^2\,. $$
Recall that
\begin{equation*}
  1 - \lambda_A = \min \Big\{\frac{1}{4d}\sum_{x \sim y} ( g(x) -  g(y))^2 :  |g|^2_2 = 1,g(x) = 0 \  \forall x \not \in A \Big\} \quad \forall A \subseteq \Z^d\,,
\end{equation*}
where $\Phi_A/|\Phi_A|_2$ is the minimizer.
Therefore, we have
\begin{align*}
	1 - \lambda_{D_1 \setminus D_2} \leq &\frac{\frac{1}{4d}\sum_{x \sim y}( \tilde \Phi_{D_1}(x) -  \tilde \Phi_{D_1}(y))^2 }{|\tilde \Phi_{D_1}|_2^2} \leq \frac{\frac{1}{4d}\sum_{x \sim y} (  \Phi_{D_1}(x) -   \Phi_{D_1}(y))^2 + q |\Phi_{D_1}|_2^2 }{|\Phi_{D_1}|_2^2 (1 - q)}\\
	& =\frac{1 - \lambda_{D_1} + q}{1-q} =  1- \lambda_{D_1} + \frac{2q - q\lambda_{D_1}}{1-q}\,.
\end{align*}
Since $\lambda_{D_1} \geq 0$, this yields the desired result.
\end{proof}

\begin{lemma}
\label{RemoveB}	Let $B_{R_1},B_{R_2},B_{R_3}$ be three concentric balls whose radii $R_1 >R_2 >R_3$ are sufficiently large, and let $\tp$, $\tcp$ be finite subsets of $\Z^d$. Suppose $B_{R_1} \subset \tp$; and suppose that $\tcp$ can be obtained from $\tp$ by removing some points in $B_{R_2}$, that is,
	$$ \tcp \subset \tp\quad and \quad \tp \setminus \tcp \subset B_{R_2}\,.$$
In addition, we assume that $b_1,b_2>0$ satisfy\begin{align}
	\label{eq:eigenball}
		\sum_{x \in B_{R_3}}\Phi_{\tp}(x), \sum_{x \in B_{R_3}}\Phi_{\tcp}(x) \geq \consa_1&\,,\\
	\label{eq:eigenbd}
	  and \quad \lambda_{\tcp} \geq 1 - \consa_2 R_1^{-2}\,.\quad \quad&
	 \end{align}
	 Then there exist constants $c_\consa  = c_\consa(\consa_1,\consa_2,d)>0$ and $C_d = C_d(d)>0$ such that
	 \begin{equation}
	 \label{eq:removeB}
	 	\lambda_{\tp} - \lambda_{\tcp} \geq \frac{c_\consa}{R_1^d(\log R_1)^{\11_{d=2}}}\Big(1 - \frac{R_2}{R_1} \Big)^{C_d} |\tp \setminus \tcp|^{(d-2)/d}\,.
	 \end{equation}
\end{lemma}
\begin{proof}
%We denote $\tp = \plr \cup B_{\delta,\lj-1}$ and $\tcp = \plr \cup (B_{\delta,\lj-1} \setminus B_{\delta,\lj})$.
 We first see that by $\tcp \subset \tp$,
 \begin{align*}
 	(\lambda_{\tp} - \lambda_{\tcp}) \langle \Phi_{\tp},\Phi_{\tcp}\rangle&=\langle\Phi_{\tp}, (P|_{\tp} - P|_{\tcp}) \Phi_{\tcp}\rangle\\
 	& =\frac{1}{2d} \sum_{x \sim y, x\in \tp \setminus \tcp}\Phi_{\tp}(x)\Phi_{\tcp}(y)\,.
 \end{align*}
Since $\tp \setminus \tcp \subset B_{R_2}$, it follows that
\begin{equation}
\label{eq:Flu95}
	\lambda_{\tp} - \lambda_{\tcp} \geq  \frac{\min_{x \in B_{R_2}}\Phi_{\tp}(x)}{2d\,|\Phi_{\tp}|_2|\Phi_{\tcp}|_2} \sum_{y \in \partial (\tp \setminus \tcp)}\Phi_{\tcp}(y)\,.
\end{equation}
First we give a lower bound on $\Phi_\tp$ on $B_{R_2}$. Note that for all $x \in B_{R_2}$ and $y \in B_{R_3}$, by \cite[Proposition 6.9.4]{LL10} and taking into account the periodicity of the random walk, we have $$p^{B_{R_1}}_{R_1^2}(y,x) + p^{B_{R_1}}_{R_1^2+1}(y,x) \geq \frac{c}{R_1^d} \Big(1 - \frac{R_2}{R_1}\Big) \Big(1 - \frac{R_3}{R_1}\Big) \,.$$
 Combined with \eqref{eq:eigenball} and $\lambda_\tp \leq 1$, it yields that for all $x \in B_{R_2}$
\begin{equation}
\label{eq:Flu95+1}
\begin{split}
\Phi_{\tp}(x) &=\frac{1}{2} \sum_{i = R_1^2,R_1^2 + 1} \lambda_{\tp}^{-i}\sum_{y \in \tp}\Phi_\tp(y)\Pr^{y}(S_{i} = x, \tau_{\tp^c} >i)\\
	& \geq \frac{1}{2}\sum_{y \in B_{R_3}}\Phi_\tp(y) \cdot \frac{c}{R_1^d} \Big(1 - \frac{R_2}{R_1}\Big) \Big(1 - \frac{R_3}{R_1}\Big)\\
	&\geq \frac{cb_1}{2R_1^d} \Big(1 - \frac{R_2}{R_1}\Big)^2 \,.
\end{split}
\end{equation}
On the other hand, combining \eqref{eq:eigenbd} and Lemma \ref{Philinfbd} yields
\begin{equation}
\label{eq:Flu95+2}
	|\Phi_{\tp}|_2^{2} \leq Cb_2^{d/2}R_1^{-d}, \quad |\Phi_{\tcp}|_2^{2}  \leq Cb_2^{d/2}R_1^{-d}\,.
\end{equation}
Combining \eqref{eq:Flu95}, \eqref{eq:Flu95+1}, and \eqref{eq:Flu95+2}, we see that to prove \eqref{eq:removeB}, it suffices to prove
\begin{equation}
 	 \label{eq:hp-step1}
	\sum_{x \in \partial(\tp \setminus \tcp)} \Phi_{\tcp}(x) \geq \frac{c_b}{R_1^{d}(\log R_1)^{\11_{d=2}}}\Big(1 - \frac{R_2}{R_1}\Big)^{C_d} |\tp \setminus \tcp|^{(d-2)/d}\,,
\end{equation}
where $c_b  = c_b(\consa_1,\consa_2,d)$ and $C_d = C_d(d)$ are positive constants, with $C_d$ to be chosen later in \eqref{eq:chaining}.

To verify \eqref{eq:hp-step1}, we first consider the case
\begin{equation}
\label{eq:ll2}
\sum_{x \in \tcp}\Phi_{\tcp}(x)\Pr^{x}(S_{\tau_{\tcp^c}} \in \tp \setminus \tcp)  > {\frac{b_1}{2}} \Big(1 - \frac{R_2}{R_1}\Big)^{C_d} \,.
\end{equation}
Note that
\begin{equation}
\label{eq:73111}
\begin{split}
		\sum_{x \in \tcp}\Phi_{\tcp}(x)\Pr^{x}(S_{\tau_{\tcp^c}} \in \tp \setminus \tcp) &\leq \sum_{i \geq 0} \sum_{x \in \tcp}\Phi_{\tcp}(x)\Pr^{x}(\tau_{\tcp^c} >i, S_i \in \partial (\tp \setminus \tcp))\\
	&=\sum_{i \geq 0} \lambda_{\tcp}^i \sum_{x \in \partial (\tp \setminus \tcp)}\Phi_{\tcp}(x)\\
	& = \frac{1}{1 - \lambda_{\tcp}} \sum_{x \in \partial (\tp \setminus \tcp)}\Phi_{\tcp}(x)\,.
\end{split}
\end{equation}
On the other hand, \eqref{eq:eigenball} implies $|\Phi_{\tcp}|_\infty \geq c b_1R_1^{-d}$. Then Lemma \ref{Philinfbd} implies $$\lambda_{\tcp} \leq 1- c b_1^{2/d}R_1^{-2}\,.$$
Substituting this into \eqref{eq:73111} and using the assumption that \eqref{eq:ll2} and the fact that $|\tp \setminus \tp^{\circ}| \leq (2 R_1)^d$, we obtain \eqref{eq:hp-step1}.

Now we consider the other case
\begin{equation}
\label{eq:ll1}
\sum_{x \in \tcp}\Phi_{\tcp}(x)\Pr^{x}(S_{\tau_{\tcp^c}} \in \tp \setminus \tcp)  \leq \frac{b_1}{2}\Big(1 - \frac{R_2}{R_1}\Big)^{C_d}\,,
\end{equation}
in which case the probability of exiting $\tcp$ via $\tp \setminus \tcp \subset B_{R_2}$ is small. Heuristically, this allows us to approximate the random walk in $\tcp$ by the walk in $\tp$, and then we are able to get good control on the Green's function which relates to the eigenfunction $\Phi_{\tcp}$ as follows. For any $x \in \tcp$, by the eigenvalue equation for the resolvent, we have
\begin{equation}
\label{eq:phi-decom}
	 \Phi_{\tcp}(x) = (1 - \lambda_{\tcp})\sum_v\Phi_{\tcp}(v)G_{\tcp}(v,x)\,,
\end{equation}
where
\begin{equation}
	G_{\tcp}(v,x): = \sum_{t = 0}^\infty \P^v(S_t = x, \tau_{\tcp^c}\geq t)\,.
\end{equation}
We will get a lower bound on $\Phi_{\tcp}(x)$ by restricting the sum over $v$ in \eqref{eq:phi-decom} to the annulus  $$\mathcal A : = B_{R_1 - (R_1-R_2)/3} \setminus B_{R_1 - 2(R_1-R_2)/3} \subset \tcp\,.$$
For all $v \in \mathcal A$ and $u \in B_{R_3}$,
\begin{equation}
\label{eq:123345rfd}
 	G_{\tcp}(u,v) \geq G_\tp(u,v)- \Pr^u(S_{\tau_{\tcp^c}} \in\tp \setminus \tp^{\circ}) \cdot \max_{y \in \tp \setminus \tcp}G_\tp(y,v)\,.
 \end{equation}
Since the function $x \mapsto G_\tp(x,v)$ is harmonic on $B_{R_1-2(R_1-R_2)/3}$, we can use the Harnack inequality and a standard chaining argument as in~\cite[(4.62)]{DFSX18} to obtain
that for a constant $C_d$ depending only on $d$,
\begin{equation}
\label{eq:chaining}
	\max_{y \in B_{R_2}}\frac{G_\tp(y,v)}{G_\tp(u,v)} \leq C_d \Big(1 - \frac{R_2}{R_1}\Big)^{-C_d}\,.
\end{equation}
By the strong Markov property at time $\tau_v$ and \cite[Proposition 6.9.4]{LL10}, we get
\begin{equation}
\label{eq:rw293r}
\begin{split}
	G_\tp(u,v)  \geq G_{B_{R_1}}(u,v)\geq c \Big(1 - \frac{R_2}{R_1}\Big)\Big(1 - \frac{R_3}{R_1}\Big) R_1^{-d+2}\,.
\end{split}
\end{equation}
Combining \eqref{eq:123345rfd}, \eqref{eq:chaining}, and \eqref{eq:rw293r} yields
 $$ G_{\tcp}(u,v) \geq c  \Big(1 - \frac{R_2}{R_1}\Big)^2R_1^{-d+2}\Big[1 -    \Big(1 - \frac{R_2}{R_1}\Big)^{-C_d}\Pr^u(\tau_{\tcp^c} \in\tp \setminus \tp^{\circ}) \Big]\,.$$
Combining with \eqref{eq:eigenball} and \eqref{eq:ll1}, we get for all $v \in \mathcal A$ and $u \in B_{R_3}$,
\begin{equation*}
	\begin{split}
		&\sum_{u \in B_{R_3}}\Phi_{\tcp}(u)G_{\tcp}(u,v)  \\
	\geq& c \Big(1 - \frac{R_2}{R_1}\Big)^2R_1^{-d+2}\Big(\sum_{u \in B_{R_3}}\Phi_{\tcp}(u) -    \Big(1 - \frac{R_2}{R_1}\Big)^{-C_d}\sum_{u \in B_{R_3}}\Phi_{\tcp}(u)\Pr^u(\tau_{\tcp^c} \in\tp \setminus \tp^{\circ}) \Big)\\
		\geq& c b_1\Big(1 - \frac{R_2}{R_1}\Big)^2R_1^{-d+2}\,.
	\end{split}
\end{equation*}
Therefore, by \eqref{eq:phi-decom} and \eqref{eq:eigenbd} we get for all $v \in \mathcal A$,
\begin{equation*}
	\begin{split}
	\Phi_{\tcp}(v) &\geq 	(1 - \lambda_{\tcp})\sum_{u \in B_{R_3}}\Phi_{\tcp}(u)G_{\tcp}(u,v) \geq cb_1b_2\Big(1 - \frac{R_2}{R_1}\Big)^2R_1^{-d}\,.
	\end{split}
\end{equation*}
 Then by \eqref{eq:phi-decom} and \eqref{eq:eigenbd} again (summing over $v \in \mathcal A$), we get for $x \in  \tcp$,
\begin{equation*}
\begin{split}
\Phi_{\tcp}(x) &\geq c(1-\lambda_{\tcp})\cdot b_1b_2\Big(1 - \frac{R_2}{R_1}\Big)^2R_1^{-d} \sum_{v \in \mathcal A}
 G_{\tcp}(v,x)\\
   &\geq cb_1b_2^2\Big(1 - \frac{R_2}{R_1}\Big)^2 R_1^{-d-2}\Ex^x\big[|\{1 \leq i \leq \tau_{\tcp^c}: S_i \in \mathcal A \}|;{\tau_{\tp \setminus \tcp} >\tau_{B_{R_1}^c}}\big]\,.	
\end{split}
\end{equation*}
Conditioned on $\tau_{\tp \setminus \tcp} >\tau_{B_{R_1}^c}$, the random walk must cross $\mathcal A$. Uniformly in starting and ending points, the first crossing of $\mathcal A$ has length at least $c(R_1 - R_2)^2$ with positive probability. Hence, we get
\begin{equation}
\label{eq:phi-tcp-x}
	\Phi_{\tcp}(x) \geq cb_1b_2^2\Big(1 - \frac{R_2}{R_1}\Big)^4  R_1^{-d}\Pr^x(\tau_{\tp \setminus \tcp} >\tau_{B_{R_1}^c})\,.
\end{equation}

For $d \geq 3$, summing over $x \in \partial(\tp \setminus \tcp)$ in \eqref{eq:phi-tcp-x} gives
\begin{equation}
	\sum_{x \in \partial(\tp \setminus \tcp)} \Phi_{\tcp}(x) \geq c\Big(1 - \frac{R_2}{R_1}\Big)^4  R_1^{-d}\capa(\tp \setminus \tcp)\,,
\end{equation}
where for $d \geq 3$,
\begin{equation}
\capa(\tp \setminus \tcp): = \sum_{x \in \tp \setminus \tcp} \Pr^{x}(S_t \not \in \tp \setminus \tcp, \text{for all }t \geq 1	)\,.
\end{equation}
Combining with $\capa(\tp \setminus \tcp) \geq c |\tp \setminus \tcp|^{(d-2)/d}$ (see the proof of \cite[Proposition 2.5.1]{L13}) yields \eqref{eq:hp-step1}.

For $d = 2$, fix an arbitrary $z \in \tp \setminus \tcp$. By decomposing a random walk path that starts from $z$ and exists $B_{R_1}$ before returning to $z$ according to its last exit time from $\tp\setminus\tcp$, we have
\begin{align*}
	\Pr^{z}(\tau_z^+ > \tau_{B^c_{R_1}}) &= \sum_{x \in \partial(\tp \setminus \tcp)}\Pr^{x}(S_1 \in \tp \setminus \tcp, \tau_{z} <\tau_{B^c_{R_1}})\Pr^{x}(\tau_{\tp \setminus \tcp} >\tau_{B_{R_1}^c}) \\
	&\leq \sum_{x \in \partial(\tp \setminus \tcp)}\Pr^{x}(\tau_{\tp \setminus \tcp} >\tau_{B_{R_1}^c}) \,,
\end{align*}
where $\tau_z^+ := \inf\{ t \geq 1: S_t = z\}$.
Combining with \eqref{eq:phi-tcp-x} and \cite[Proposition 6.4.3]{LL10}, which implies that for $R_1$ sufficiently large,
\begin{equation}
\label{eq:A.21}
	\Pr^{z}(\tau_z^+ > \tau_{B^c_{R_1}})\geq c(\log R_1)^{-1}\,,
\end{equation}
we get \eqref{eq:hp-step1}. We thus complete the proof of \eqref{eq:removeB}.
\end{proof}

\section{An isoperimetric inequality}
The following isoperimetric inequality is needed in the proof of Lemma \ref{iter}. It says that if we partition a ball in $\Z^d$ into two parts, then the area of the interface between the two parts can be bounded from below as a function of the volume of the smaller part.
\begin{lemma}
\label{isolemma}
Fix an arbitrary $R\geq 1$. Let $\mathbf{B}(0,R) = \{x \in \R^d: |x|_2 \leq R\}$, and let $B := \mathbf{B}(0,R) \cap \Z^d$. Suppose $B = A_1 \cup A_2$ is a partition of $B$. Then
\begin{equation}
\label{eq:isolemma}
	\min\{ |\partial A_1 \cap A_2|,| \partial A_2 \cap A_1|\} \geq c \min(|A_1|, |A_2|)^{1-1/d}\,,
\end{equation}
where $c>0$ is a constant depending only on $d$, and recall that $\partial A_i := \{x\in A_i^c: |x-y|_1 = 1 \mbox{ for some } y\in A_i\}$ for $i=1,2$.
\end{lemma}
\begin{proof}
For any $x \in \Z^d$, let $x^* := [x-1/2,x+1/2]^d\cap \mathbf{B}(0,R)$. Then there exist constants $c,C>0$ depending only on $d$ such that for all $x \in B$,
$$ \mathrm{vol}(x^*) \geq c \quad \text{and} \quad \mathrm{suf}(x^*) \leq C\,,$$
where $\mathrm{vol}(x^*)$ and $\mathrm{suf}(x^*)$ are volume and surface area of $x^*$, respectively.
For any set $U \subset \Z^d$, we denote $U^* := \bigcup_{x \in U}x^*$. If we denote
$$ q: = \min\{\mathrm{vol}(A_1),\mathrm{vol}(A_2)\}\,,$$
then
$$ q^*:=\min\{\mathrm{vol}(A_1^*),\mathrm{vol}(A_2^*)\} \geq cq\,.$$
By the isoperimetric inequality in $\R^d$,
$$\mathrm{suf}(A_1^*) +  \mathrm{suf}(A_2^*) \geq d \mathrm{vol}(\mathbf{B}(0,1))^{1/d}  \cdot \big[\mathrm{vol}(A_1^*)^{1- 1/d}+\mathrm{vol}(A_2^*)^{1- 1/d}\big]\,.$$
Since for any $\alpha \in (0,1)$, the function $x^\alpha - (x+1)^\alpha$ is increasing in $x \in (0,+\infty)$, we have
$$x^{1-1/d} + y^{1-1/d}\geq (x+y)^{1-1/d} + (2 - 2^{1-1/d})y^{1-1/d}\text{ for all }0 \leq y \leq x\,.$$ Therefore,
\begin{align*}
	\mathrm{suf}(A_1^*) +  \mathrm{suf}(A_2^*) &\geq d \mathrm{vol}(\mathbf{B}(0,1))^{1/d}  \big[\mathrm{vol}(B^*)^{1-1/d} + (2 - 2^{1-1/d}){q^*}^{1-1/d}\big]\\
	&\geq  \mathrm{suf}(B^*)+  d \mathrm{vol}(\mathbf{B}(0,1))^{1/d}  (2 - 2^{1-1/d})(cq)^{1-1/d}\,.
\end{align*}
Note that the interface between $A_1^*$ and $A_2^*$ is contained in both the surface of $(\partial A_1 \cap B)^*$ and $(\partial A_2 \cap B)^*$, and it is counted exactly twice in $\mathrm{suf}(A_1^*) +  \mathrm{suf}(A_2^*)  - \mathrm{suf}(B^*) $. It follows that
\begin{align*}
	\mathrm{suf}(A_1^*) +  \mathrm{suf}(A_2^*)  - \mathrm{suf}(B^*) &\leq 2\min\Big\{\mathrm{suf}\big((\partial A_1 \cap B)^*\big),\mathrm{suf}\big((\partial A_2 \cap B)^*\big)\Big\}\\
	& \leq C\min\big(|\partial A_1 \cap B|,|\partial A_2 \cap B| \big) \\
	& = C\min\big(|\partial A_1 \cap A_2|,|\partial A_2 \cap A_1| \big) \,.
\end{align*}
Combining the previous two inequalities completes the proof of \eqref{eq:isolemma}.
\end{proof}

\section*{Acknowledgements}
J.~Ding is supported by NSF grant DMS-1757479 and an Alfred Sloan fellowship.
R.~Fukushima is supported by JSPS KAKENHI Grant Number 16K05200 and ISHIZUE 2019 of Kyoto University Research Development Program.
R.~Sun is supported by NUS Tier 1 grant R-146-000-288-114.

% \small
% \bibliography{Qnh}
% \bibliographystyle{abbrv}

\end{document}